\documentclass[pagesize,pdftex]{scrartcl}
\usepackage{latexsym} 
\usepackage[intlimits]{amsmath}
\usepackage{amsthm}
\usepackage{amsfonts}
\usepackage{amssymb}
\usepackage{amsxtra}
\usepackage{amscd}
\usepackage{ifthen}
\usepackage{graphicx}
\usepackage{color}
\usepackage[shortlabels]{enumitem}
\usepackage{mathrsfs}
\usepackage[pagebackref=true]{hyperref}
\usepackage[arrow, matrix, curve]{xy}

\hypersetup{
pdfauthor={Aday Celik, Mads Kyed},
pdftitle={Nonlinear Wave Equation with Damping: Periodic Forcing and Non-Resonant Solutions to the Kuznetsov Equation},
breaklinks=true,
colorlinks=true,
linkcolor=blue,
citecolor=blue,
urlcolor=blue,
filecolor=blue,
}

\numberwithin{equation}{section} 
\setkomafont{title}{\normalfont}

\newenvironment{pdeq}{ \left\{ \begin{aligned}}{\end{aligned}\right.}

\newcommand{\np}[1]{(#1)}
\newcommand{\nb}[1]{[#1]}
\newcommand{\bp}[1]{\big(#1\big)}
\newcommand{\bb}[1]{\big[#1\big]}
\newcommand{\Bp}[1]{\bigg(#1\bigg)}

\newcommand{\calk}{{\mathcal K}}

\newcommand{\calp}{{\mathcal P}}

\newcommand{\calt}{{\mathcal T}}

\newcommand{\R}{\mathbb{R}}
\newcommand{\C}{\mathbb{C}}
\newcommand{\Z}{\mathbb{Z}}

\newcommand{\N}{\mathbb{N}}
\DeclareMathOperator{\e}{e}

\DeclareMathOperator{\T}{T}

\DeclareMathOperator{\supp}{supp}
\DeclareMathOperator{\graph}{graph}

\newcommand{\TD}{\operatorname{Tr}_D}
\newcommand{\TN}{\operatorname{Tr}_N}
\newcommand{\ra}{\rightarrow}

\newcommand{\set}[1]{\ensuremath{\{#1\}}}

\newcommand{\setc}[2]{\ensuremath{\{#1\ \lvert\ #2\}}}

\newcommand{\closure}[2]{\overline{#1}^{#2}}
\newcommand{\seqN}[1]{\ensuremath{\set{#1}_{n=1}^\infty}}

\newcommand{\proj}{\calp}
\newcommand{\projcompl}{\calp_\bot}

\newcommand{\quotientmap}{\pi}
\newcommand{\grp}{G}
\newcommand{\dualgrp}{\widehat{G}}

\newcommand{\grpH}{H}
\newcommand{\dualgrpH}{\widehat{H}}
\newcommand{\torus}{{\mathbb T}}

\newcommand{\grad}{\nabla}

\newcommand{\pdn}[1]{\frac{\partial #1}{\partial n}}

\newcommand{\dx}{{\mathrm d}x}

\newcommand{\dg}{{\mathrm d}g}
\newcommand{\ds}{{\mathrm d}s}
\newcommand{\dt}{{\mathrm d}t}

\newcommand{\dS}{{\mathrm d}S}
\newcommand{\dxi}{{\mathrm d}\xi}

\newcommand{\SR}{\mathscr{S}}

\newcommand{\TDR}{\mathscr{S^\prime}}

\newcommand{\TDRper}{\mathscr{S^\prime_\bot}}

\newcommand{\ft}[1]{\widehat{#1}}
\newcommand{\ift}[1]{#1^\vee}
\newcommand{\FT}{\mathscr{F}}
\newcommand{\iFT}{\mathscr{F}^{-1}}

\newcommand{\PR}{\mathcal{P}}
\newcommand{\oPR}{\mathcal{P}_\bot}
\newcommand{\norm}[1]{\lVert#1\rVert}

\newcommand{\snorm}[1]{{\lvert #1 \rvert}}

\newcommand{\opnorm}[1]{{\vert\kern-0.25ex\vert\kern-0.25ex\vert #1 \vert\kern-0.25ex\vert\kern-0.25ex\vert}}
\newcommand{\WSR}[2]{W^{#1,#2}}
\newcommand{\WSRD}[2]{\dot{W}^{#1,#2}}

\newcommand{\CR}[1]{C^{#1}}  
\newcommand{\LR}[1]{L^{#1}}

\newcommand{\LRloc}[1]{L^{#1}_{loc}} 
\newcommand{\CRi}{\CR \infty}
\newcommand{\CRci}{\CR \infty_0}

\newcommand{\SolS}[1]{X_\bot^{#1}}

\newcommand{\TPD}{T^p_D}
\newcommand{\TPN}{T^p_N}
\newcommand{\TPDper}{T^p_{D,per}}
\newcommand{\TPNper}{T^p_{N,per}}
\newcommand{\LRper}[1]{L^{#1}_{\mathrm{per}}}
\newcommand{\WSRper}[2]{W^{#1,#2}_{\mathrm{per}}} 

\newcommand{\CRiper}{\CR{\infty}_{\mathrm{per}}}

\newcommand{\wvel}{w}

\newcommand{\uvel}{u}
\newcommand{\us}{\uvel_s}
\newcommand{\up}{\uvel_p}

\newcommand{\tin}{\text{in }}
\newcommand{\tif}{\text{if }}
\newcommand{\ton}{\text{on }}

\renewcommand{\epsilon}{\varepsilon}

\newcommand{\tay}{\calt}
\newcommand{\per}{\tay}

\newcommand{\dualrho}{\hat{\rho}}

\newcommand{\newCCtr}[2][d]{
\newcounter{#2}\setcounter{#2}{0}
\expandafter\xdef\csname kyedtheconst#2\endcsname{#1}
}
\newcommand{\Cc}[2][nolabel]{
\stepcounter{#2}
\expandafter\ensuremath{\csname kyedtheconst#2\endcsname_{\arabic{#2}}}
\ifthenelse{\equal{#1}{nolabel}}
{}
{\expandafter\xdef\csname kyedconst#1\endcsname
{\expandafter\ensuremath{\csname kyedtheconst#2\endcsname_{\arabic{#2}}}}}
}
\newcommand{\Ccn}[2][nolabel]{
\expandafter\ensuremath{\csname kyedtheconst#2\endcsname}
\ifthenelse{\equal{#1}{nolabel}}
{}
{\expandafter\xdef\csname kyedconst#1\endcsname
{\expandafter\ensuremath{\csname kyedtheconst#2\endcsname}}}
}
\newcommand{\CcSetCtr}[2]{
\setcounter{#1}{#2}
}
\newcommand{\Cclast}[1]{
\expandafter\ensuremath{\csname kyedtheconst#1\endcsname_{\arabic{#1}}}
}
\newcommand{\Ccllast}[1]{
\addtocounter{#1}{-1}
\expandafter\ensuremath{\csname kyedtheconst#1\endcsname_{\arabic{#1}}}
\addtocounter{#1}{1}
}
\newcommand{\const}[1]{
\expandafter{\ifcsname kyedconst#1\endcsname
  \csname kyedconst#1\endcsname
\else
  \errmessage{Undefined Kyedconstant #1.}%
\fi}
}

\theoremstyle{plain}
\newtheorem{thm}{Theorem}[section]

\newtheorem{lem}[thm]{Lemma}

\newtheorem{cor}[thm]{Corollary}
\theoremstyle{remark}
\newtheorem{rem}[thm]{Remark}

\begin{document}
\title{Nonlinear Wave Equation with Damping: Periodic Forcing and Non-Resonant Solutions to the Kuznetsov Equation}

\author{
Aday Celik\\ 
Fachbereich Mathematik\\
Technische Universit\"at Darmstadt\\
Schlossgartenstr. 7, 64289 Darmstadt, Germany\\
Email: {\texttt{celik@mathematik.tu-darmstadt.de}}
\and
Mads Kyed\\ 
Fachbereich Mathematik\\
Technische Universit\"at Darmstadt\\
Schlossgartenstr. 7, 64289 Darmstadt, Germany\\
Email: {\texttt{kyed@mathematik.tu-darmstadt.de}}
}

\date{\today}
\maketitle

\begin{abstract}
Existence of non-resonant solutions of time-periodic type are established for the Kuznetsov equation with a periodic forcing term. The
equation is considered in a three-dimensional whole-space, half-space and bounded domain, and with both non-homogeneous Dirichlet and 
Neumann boundary values. A method based on $\LR{p}$ estimates of the corresponding linearization, namely the wave equation with Kelvin-Voigt
damping, is employed.   

\end{abstract}

\noindent\textbf{MSC2010:} Primary 35L05, 35B10, 35B34.\\
\noindent\textbf{Keywords:} wave equation, damping, periodic solutions, maximal regularity.

\newCCtr[C]{C}
\newCCtr[M]{M}
\newCCtr[B]{B}
\newCCtr[\epsilon]{eps}
\CcSetCtr{eps}{-1}

\section{Introduction}

In a damped hyperbolic system with periodic forcing, resonance can be avoided if the energy from the external forces accumulated over a period is
dissipated via the damping mechanism. The existence of a time-periodic solution would be a manifestation hereof. Our aim in this article is 
to develop a method that can be used to ensure existence of this particular type of non-resonant solution for nonlinear hyperbolic systems. 
Although the method is generic in nature, we restrict our analysis to a nonlinear wave equation with  
Kelvin-Voigt damping in a three-dimensional domain. Specifically, we consider the Kuznetsov equation, which is a
nonlinear wave equation that describes acoustic wave propagation. As our main result, we show for any periodic forcing term that is sufficiently 
restricted in ``size'' existence of a time-periodic solution. 
We shall treat non-homogeneous boundary values of both Dirichlet and Neumann type.
We consider spatial domains $\Omega\subset\R^3$ that are either bounded, the half-space or the whole-space.  

Well-posedness of the initial-value problem for the Kuznetsov equation has only recently been established \cite{KaltenbacherLasiecka2011,KaltenbacherLasiecka2012,MeyerWilke2013}. Our result can be viewed as an extension of these results to the corresponding time-periodic problem. Related time-periodic problems have been studied by other authors over the years.
In particular we mention the work of \textsc{Kokocki} \cite{Kokocki2015}, where a class of nonlinear wave equations with Kelvin-Voigt damping, which do not contain the Kuznetsov equation though, are investigated.  

We will work in a setting of time-periodic functions and therefore take the whole of $\R$ as a time-axis. In the following, $(t,x)\in\R\times\Omega$ will always denote a time-variable $t$ and spatial variable $x$, respectively.  
The Kuznetsov equation with Dirichlet boundary condition then reads
\begin{align}\label{KuznetsovD}
\begin{pdeq}
\partial_{t}^2\uvel-\Delta\uvel - \lambda\partial_t\Delta\uvel - \partial_t\bp{\gamma\np{\partial_t\uvel}^2+\snorm{\grad\uvel}^2}&= f && \tin\R\times\Omega, \\
\uvel &= g && \ton\R\times\partial\Omega.
\end{pdeq}\tag{KD}
\end{align}
The corresponding Neumann problem reads
\begin{align}\label{KuznetsovN}
\begin{pdeq}
\partial_{t}^2\uvel-\Delta\uvel - \lambda\partial_t\Delta\uvel - \partial_t\bp{\gamma\np{\partial_t\uvel}^2+\snorm{\grad\uvel}^2}&= f && \tin\R\times\Omega, \\
\pdn\uvel &= g && \ton\R\times\partial\Omega.
\end{pdeq}\tag{KN}
\end{align}
Here, $\lambda,\gamma>0$ are constants. We shall consider both data and solutions that are time-periodic with the same period $\per>0$, that is, functions
$\uvel$, $f$ and $g$ satisfying
\begin{align}\label{tp_cond}
\forall t\in\R:\quad h(\per+t,\cdot)=h(t,\cdot).\tag{TP}
\end{align}
More precisely, we will show for data $f$ and $g$ satisfying \eqref{tp_cond}, and whose norm in appropriate Sobolev spaces are sufficiently small,
the existence of a solution $\uvel$ to \eqref{KuznetsovD} and \eqref{KuznetsovN} also satisfying \eqref{tp_cond}.  

Our approach is based on $\LR{p}$ estimates of solutions to the corresponding linearizations 
\begin{align}\label{damp_waveD}
\begin{pdeq}
\partial_{t}^2\uvel-\Delta\uvel - \lambda\partial_t\Delta\uvel &= f && \tin\R\times\Omega, \\
\uvel &= g && \ton\R\times\partial\Omega
\end{pdeq}\tag{WD}
\end{align}
and
\begin{align}\label{damp_waveN}
\begin{pdeq}
\partial_{t}^2\uvel-\Delta\uvel - \lambda\partial_t\Delta\uvel &= f && \tin\R\times\Omega, \\
\pdn\uvel &= g && \ton\R\times\partial\Omega,
\end{pdeq}\tag{WN}
\end{align}
of \eqref{KuznetsovD} and \eqref{KuznetsovN}, respectively, and an application of the contraction mapping principle. The novelty of our approach is rooted in the method we employ to establish the $\LR{p}$ estimates of \eqref{damp_waveD} and \eqref{damp_waveN}.   
Instead of relying on a Poincar\'{e} map, which is the standard procedure in the investigation of time-periodic problems, and also the approach
used in \cite{Kokocki2015}, we obtain the estimates directly via a representation formula for the solution. 
We hereby circumvent completely the theory for the corresponding initial-value problem that is needed to construct 
a Poincar\'{e} map. Not only do we develop a much more direct approach, the representation formula we establish also seems interesting in the context of resonance, or rather the avoidance hereof, since it exposes the way different modes of the solution are damped in relation to the modes of the forcing term. We shall briefly outline the method in the whole-space case $\Omega = \R^3$. The main idea is to reformulate the time-periodic problem as a partial differential equation on the locally compact abelian group $\grp:=\R/\per\Z\times\R^3$. Since the data $f$ and the solution $\uvel$ are both $\per$-time-periodic functions, they can be interpreted as functions on $\grp$. A differentiable structure on $\grp$ is canonically inherited from $\R\times\R^3$ via the quotient mapping $\pi:\R\times\R^3\to\R/\per\Z\times\R^3$ in such a way that the damped wave equation can be reformulated 
as a partial differential equation
\begin{align}\label{dw_group}
\partial_{t}^2\uvel-\Delta\uvel - \lambda\partial_t\Delta\uvel &= f \quad \tin\grp
\end{align}
in a setting of functions $\uvel:\grp\to\R$ and $f:\grp\to\R$. In this setting, it is possible to use the Fourier transform $\FT_\grp$ in combination with the space of tempered distributions $\TDR(\grp)$, the dual of the Schwartz-Bruhat space $\SR(\grp)$, and derive from \eqref{dw_group} the representation formula
\begin{align}\label{Lsg}
\uvel = \iFT_\grp\left[\frac{1}{|\xi|^2 - k^2 + i\lambda k|\xi|^2}\FT_\grp\nb{f}\right]
\end{align}
when $f\in\SR(\grp)$. Here $\left(k, \xi \right)$ denote points in the dual group $\dualgrp:=\frac{2\pi}{\per}\Z\times\R^3$. 
The term $i\lambda k|\xi|^2$ in the denominator of the Fourier multiplier in \eqref{Lsg} stems from the damping. 
For modes $k\neq 0$, the multiplier is bounded due to the damping term,  whereas the mode $k=0$ of the multiplier is not ``damped'' at all. To obtain the desired estimates of $\uvel$, we shall therefore split the ``damped'' and ``non-damped'' modes, in this case 
\begin{align}\label{split}
\uvel = 
\iFT_\grp\left[\frac{1}{|\xi|^2}\FT_\grp\nb{f}\right]+
\iFT_\grp\left[\frac{(1 - \delta_\Z(k))}{|\xi|^2 - k^2 + i\lambda k|\xi|^2}\FT_\grp\nb{f}\right] =: \us+\up,
\end{align}
where $\delta_\Z(k):=1$ if $k=0$ and $\delta_\Z(k):=0$ otherwise. 
The main advantage of the decomposition is that the bounded multiplier in the representation of $\up$ leads to a better $\LR{p}$ estimate 
than can be obtained for the full solution $\uvel$.
We shall establish the estimate by invoking a transference principle for group multipliers, which allows us to transfer the multiplier into a 
Euclidean setting. This principle was originally established by \textsc{de Leeuw} \cite{Leeuw1965} and later generalized by \textsc{Edwards} and \textsc{Gaudry} \cite{EdwardsGaudry}. The estimate of $\us$ can be obtained by standard methods. 
Clearly, a more complex damping than the Kelvin-Voigt damping term would lead to a more complex splitting in \eqref{split}, but the general idea
should still be applicable. We postpone the investigation of more general damping mechanisms to future works. 

\section{Preliminaries}\label{pre}

\subsection{Notation}

In the following, $\Omega\subset\R^3$ will always denote a domain, namely, an open connected set.
Points in $\R\times\Omega$ are generally denoted by $(t,x)$, with $t$ being referred to as time and $x$ as the spatial variable. 
A time-period $\per>0$ remains fixed.

For functions $f$ defined on time-space domains, we let
\begin{equation}\label{DefOfProj}
\PR f (t,x) := \frac{1}{\per}\int_0^\per f(s,x)\,\ds,\quad \oPR f(t,x):= f(t,x)-\proj f(t,x)
\end{equation}
whenever the integral is well defined. Since $\PR f$ is independent on time $t$, we shall implicitly treat $\proj f$ as a function 
of the spatial variable $x$ only.

Classical Lebesgue and Sobolev spaces with respect to spatial domains are denoted by $\LR{p}(\Omega)$ and $\WSR{k}{p}(\Omega)$, respectively. 
We further introduce the homogeneous Sobolev space
\begin{equation*}
\WSRD{2}{p}\left(\Omega\right) := \setc{\uvel\in\LRloc{1}(\Omega)}{\partial_{x}^\alpha\uvel\in\LR{p}(\Omega), \left|\alpha\right| = 2},\quad \snorm{\uvel}_{2,p} := \Bp{\sum_{\snorm{\alpha} = 2} \norm{\partial_x^\alpha\uvel}_p^p}^{\frac{1}{p}}.
\end{equation*}

\subsection{Group setting and Fourier transform}\label{GroupSettingSection}

We introduce the group $\grp:=\torus\times\R^3$, where $\torus$ denotes the torus group $\R/\per\Z$. 
The quotient mapping 
\begin{align}\label{topology}
\quotientmap :\R\times\R^3\to\grp, \ \pi\left(t, x\right) := \left(\left[t\right], x\right)
\end{align}
induces a topology and a differentiable structure on $\grp$.
Equipped with the quotient topology, $\grp$ becomes a locally compact abelian group. Via the restriction 
$\Pi := \pi\big|_{\left[0,\per\right)\times\R^3}$, we can identify $\grp$ with the domain $\left[0,\per\right)\times\R^3$, and the Haar measure $\dg$ on $\grp$ as the product of the Lebesgue measure on $\R^3$ and the Lebesgue measure on $\left[0,\per\right)$. From the uniqueness of the Haar measure up to a constant, it is possible to choose $\dg$ in such a way that 
\begin{align*}
\int_\grp \uvel\left(g\right)\,\dg = \frac{1}{\per}\int_0^\per\int_{\R^3} \uvel\circ\Pi\left(x, t\right)\,\dx\dt.
\end{align*}
For the sake of convenience, we will omit the $\Pi$ in integrals of $\grp$-defined functions with respect to $\dx\dt$. 
Furthermore, we define by 
\begin{align}\label{DefOfSmoothFunctionsOnGrp}
\CRi(\grp) := \setc{\uvel:\grp\ra\R}{\exists U\in\CRi\bp{\R\times\R^3}:\ U=\uvel\circ\pi},
\end{align}
the space of smooth functions on $\grp$. Derivatives of a function $\uvel\in\CRi(\grp)$ are defined by $$\partial_t^\beta\partial_x^\alpha\uvel := \left[\partial_t^\beta\partial_x^\alpha\left(\uvel\circ\pi\right)\right]\circ\Pi^{-1}\quad\forall\left(\alpha, \beta\right)\in\N_0^3\times\N_0.$$ 
By $\CRci(\grp) := \setc{\uvel\in\CRi(\grp)}{\supp\uvel \text{ is compact}}$ we denote the space of compactly supported smooth functions on $\grp$.
The Schwartz-Bruhat space on $\grp$ is defined by 
\begin{align*}
\SR(\grp) := \setc{\uvel\in\CRi(\grp)}{\forall (\alpha, \beta, \gamma)\in\N_0^3\times\N_0\times\N_0^3: \ \rho_{\alpha, \beta, \gamma}(\uvel)<\infty},
\end{align*}
where
\begin{align*}
\rho_{\alpha, \beta, \gamma}(\uvel) := \sup_{(t,x)\in\grp}{\left|x^\gamma\partial_t^\beta\partial_x^\alpha\uvel(t,x)\right|}.
\end{align*}
Equipped with the semi-norm topology of the family ${\{\rho_{\alpha, \beta, \gamma} | \left(\alpha, \beta, \gamma\right)\in\N_0^3\times\N_0\times\N_0^3 \}}$, $\SR(\grp)$ becomes a topological vector space. The corresponding topological dual space $\TDR(\grp)$ equipped with the $\text{weak}*$ topology is referred to as the space of tempered distributions on $\grp$. Distributional derivatives for a tempered distribution $\uvel$ are defined by duality as in the classical case.

Let $\dualgrp$ denote the dual group of $\grp$. Each $\left(k, \xi\right)\in\frac{2\pi}{\per}\Z\times\R^3$ can be associated with a character $\chi: \grp\to\C, \ \chi\left(t, x\right):= e^{ix\cdot\xi + ikt}$ on $\grp$. Thus we can identify $\dualgrp = \frac{2\pi}{\per}\Z\times\R^3$,
and the compact-open topology on $\dualgrp$ as the product of the Euclidean topology on $\R^3$ and the discrete topology on $\frac{2\pi}{\per}\Z$. The Haar measure on $\dualgrp$ is then the product of the counting measure on $\frac{2\pi}{\per}\Z$ and the Lebesgue measure on $\R^3$. 
The space of smooth functions on $\dualgrp$ is defined as
\begin{align*}
\CRi(\dualgrp) := \setc{\uvel\in\CR{}(\dualgrp)}{\forall k\in\frac{2\pi}{\per}\Z: \uvel(k, \cdot)\in\CRi(\R^3)},
\end{align*}
and the Schwartz-Bruhat space as
\begin{align*}
\SR(\dualgrp) := \setc{\uvel\in\CRi(\dualgrp)}{\forall (\alpha, \beta, \gamma)\in\N_0^3\times\N_0^3\times\N_0: \ \dualrho_{\alpha, \beta, \gamma}(\uvel)<\infty},
\end{align*}
where
\begin{align*}
\dualrho_{\alpha, \beta, \gamma}(\uvel) := \sup_{(k,\xi)\in\dualgrp}{\left|\xi^\alpha\partial_\xi^\beta k^\gamma\uvel(k,\xi)\right|}
\end{align*}
are the generic semi-norms.

By $\FT_\grp$ we denote the Fourier transform associated to the locally compact abelian group $\grp$ defined by 
\begin{align*}
\FT_\grp:\SR(\grp)\ra\SR(\dualgrp),\quad \FT_\grp\nb{\uvel}(k,\xi):= \ft{\uvel}(k,\xi) := \frac{1}{\per}\int_0^\per\int_{\R^3} \uvel(t,x)\,\e^{-ix\cdot\xi-ik t}\,\dx\dt.
\end{align*}
We recall that $\FT_\grp:\SR(\grp)\ra\SR(\dualgrp)$ is a homeomorphism and the inverse is given by
\begin{align*}
\iFT_\grp:\SR(\dualgrp)\ra\SR(\grp),\quad \iFT_\grp\nb{\wvel}(t,x):= \ift{\wvel}(t,x) := \sum_{k\in\frac{2\pi}{\per}\Z}\,\int_{\R^3} \wvel(k,\xi)\,\e^{ix\cdot\xi+ik t}\,\dxi.
\end{align*}
By duality, $\FT_\grp$ extends to a homeomorphism $\TDR(\grp)\ra\TDR(\dualgrp)$.

\subsection{Multiplier theory}
Next, we introduce two helpful tools from harmonic analysis, which will enable us to estimate the solution to \eqref{damp_waveD} and \eqref{damp_waveN}.
Due to the lack of sufficient multiplier theory in the general group setting, we make use of the following lemma to transfer the investigation into an Euclidean stetting.
The idea goes back to \textsc{De Leeuw} \cite{Leeuw1965}; the lemma below is due to \textsc{Edwards} and \textsc{Gaudry} \cite{EdwardsGaudry}.

\begin{lem}\label{transference}
Let $\grp$ and $\grpH$ be locally compact abelian groups. Moreover, let $\Phi:\dualgrp\ra\dualgrpH$ be a continuous homomorphism and $p\in [1,\infty]$. Assume that $m\in\LR{\infty}(\dualgrpH;\C)$ is a continuous $\LR{p}$-multiplier, i.e., there is a constant $C$ such that
\begin{align*}
\forall f\in\LR{2}\left(\grpH\right)\cap\LR{p}\left(\grpH\right): \norm{\iFT_\grpH\nb{m\cdot \widehat{f}}}_p\leq C\norm{f}_p.
\end{align*}
Then $m\circ\Phi\in\LR{\infty}(\dualgrp; \C)$ is also an $\LR{p}$-multiplier with
\begin{align*}
\forall f\in\LR{2}\left(\grp\right)\cap\LR{p}\left(\grp\right): \norm{\iFT_\grp\nb{m\circ\Phi\cdot \widehat{f}}}_p\leq C\norm{f}_p.
\end{align*}
\end{lem}
\begin{proof}
See \cite[Theorem B.2.1]{EdwardsGaudry}.
\end{proof}
\begin{rem}
Applying Lemma \ref{transference} with $\grp := \torus\times\R^3$, $\grpH := \R\times\R^3$ and 
$\Phi:\frac{2\pi}{\per}\Z\times\R^3\mapsto\R\times\R^3,\ \Phi(k, \xi):=(k, \xi)$,
we are able to transform and investigate an $\LR{p}$-multiplier on $\grp$ into an $\R^4$ setting.
\end{rem}
We shall make use of the following multiplier theorem of Marcinkiewicz type:
\begin{lem}\label{Marcinkiewicz}
Let $m:\R^n\ra\C$ be a bounded function with $m\in\CR{n}(\R^n)$. Assume there is a constant $A$ such that
\begin{align}\label{Mar}
\sup_{\varepsilon\in\{0, 1\}^n}\sup_{\xi\in\R^n}\left|\xi_1^{\varepsilon_1}\cdots\xi_n^{\varepsilon_n}\partial_{\xi_1}^{\varepsilon_1}\cdots\partial_{\xi_n}^{\varepsilon_n}m(\xi)\right|\leq A.
\end{align}
Then for any $p\in(1,\infty)$ there is a constant $C$ such that
\begin{align*}
\forall f\in\LR{2}(\R^n)\cap\LR{p}(\R^n): \norm{\mathscr{F}^{-1}\nb{m\cdot\widehat{f}}}_p\leq CA\norm{f}_p,
\end{align*}
with $C=C(p)$.
\end{lem}
\begin{proof}
See \cite[Corollary 6.2.5]{Grafakos}.
\end{proof}

\subsection{Function spaces}\label{FunktionSpacesSection}

Let $E(\Omega)$ be a Banach space. By the same construction as in \eqref{DefOfSmoothFunctionsOnGrp}, we introduce the space $\CRi\bp{\torus;E(\Omega)}$ 
of smooth vector-valued functions on the torus. For $p\in(1, \infty)$ and $k\in\N_0$ we further introduce the norms 
\begin{align}\label{DefOfNormsOnTorus}
\begin{aligned}
&\norm{f}_{\LR{p}\left(\torus; E(\Omega)\right)}:=\Bp{\frac{1}{\per}\int_0^\per \norm{f(t,\cdot)}^p_{E(\Omega)} \,\dt}^{\frac{1}{p}}, \\
&\norm{f}_{\WSR{k}{p}\left(\torus; E(\Omega)\right)}:=\Bp{\sum_{\alpha = 0}^k \norm{\partial_t^\alpha f}_{\LR{p}\left(\torus; E(\Omega)\right)}^p}^{\frac{1}{p}},
\end{aligned}
\end{align}
and let
\begin{align*}
&\LR{p}\left(\torus; E(\Omega)\right):= \closure{\CRi\left(\torus; E\left(\Omega\right)\right)}{\norm{\cdot}_{\LR{p}\left(\torus; E(\Omega)\right)}},\\
&\WSR{k}{p}\left(\torus; E(\Omega)\right) := \closure{\CRi\left(\torus; E\left({\Omega}{}\right)\right)}{\norm{\cdot}_{\WSR{k}{p}\left(\torus; E(\Omega)\right)}}.
\end{align*}
Clearly $\LR{p}\left(\torus\times\Omega\right) = \LR{p}\left(\torus; \LR{p}\left(\Omega\right)\right)$. We sometimes 
write $\norm{\cdot}_p$ instead of $\norm{\cdot}_{\LR{p}\left(\torus; \LR{p}\left(\Omega\right)\right)}$ when no confusion can arise.

Recalling \eqref{DefOfProj}, we observe that $\proj$ and $\projcompl$ are complementary projections on the space $\CRi\left(\torus; E(\Omega)\right)$.
We shall employ these projections 
to decompose the Lebesgue and Sobolev spaces introduced above. Since $\PR f$ is time independent, 
we shall refer to $\PR f$ as the \emph{steady-state} part of $f$, and $\oPR f$ as the \emph{purely periodic} part.
By continuity, $\PR$ and $\oPR$ extend to bounded operators on $\LR{p}\left(\torus; E(\Omega)\right)$ and $\WSR{k}{p}\left(\torus; E(\Omega)\right)$.

We introduce the anisotropic Sobolev space
\begin{align}
&\SolS{p}\left(\torus\times\Omega\right) := \oPR\WSR{2}{p}(\torus; \LR{p}(\Omega)) \cap \oPR\WSR{1}{p}\left(\torus; \WSR{2}{p}(\Omega)\right)\\
&\norm{\uvel}_{\SolS{p}} := \Bp{\norm{\partial_t^2\uvel}_{p}^p + \norm{\uvel}_{\WSR{1}{p}\left(\torus; \WSR{2}{p}\left(\Omega\right)\right)}^p}^{\frac{1}{p}}.
\end{align}
Sobolev-Slobodecki\u{\i} spaces 
\begin{align}\label{DefOfTraceSpaces}
\begin{aligned}
&\TPD(\torus\times\partial\Omega) := \WSR{2-\frac{1}{2p}}{p}\left(\torus; \LR{p}(\partial\Omega)\right)\cap\WSR{1}{p}\left(\torus; \WSR{2-\frac{1}{p}}{p}(\partial\Omega)\right),\\
&\TPN(\torus\times\partial\Omega) := \WSR{\frac{1}{2}\left(3-\frac{1}{p}\right)}{p}\left(\torus; \LR{p}(\partial\Omega)\right)\cap\WSR{1}{p}\left(\torus; \WSR{1-\frac{1}{p}}{p}(\partial\Omega)\right),
\end{aligned}
\end{align}
are defined in the usual way using real interpolation. One may verify that the trace operators 
\begin{align*}
&\TD: \WSR{2}{p}(\torus; \LR{p}(\Omega)) \cap \WSR{1}{p}\left(\torus; \WSR{2}{p}(\Omega)\right) \to \TPD(\torus\times\partial\Omega),\quad
\TD(\uvel):=\uvel_{|{\torus\times\partial\Omega}},\\
&\TN: \WSR{2}{p}(\torus; \LR{p}(\Omega)) \cap \WSR{1}{p}\left(\torus; \WSR{2}{p}(\Omega)\right) \to \TPN(\torus\times\partial\Omega),\quad 
\TN(\uvel) :=\pdn\uvel_{|{\torus\times\partial\Omega}},
\end{align*}
are continuous and surjective; see for example \cite{DHP}.

\section{Linear Problem}\label{lin}
We shall investigate the linearized problems \eqref{damp_waveD} and \eqref{damp_waveN} and establish maximal $\LR{p}$ regularity in a setting of $\per$-time-periodic functions. 
For a Banach space $E(\Omega)$ we define by 
\begin{align*}
\CRiper\left(\R; E\left(\Omega\right)\right):= \setc{f\in\CRi\left(\R; E\left(\Omega\right)\right)}{f(t+\per, x) = f(t, x) }
\end{align*}
the space of smooth vector-valued $\per$-time-periodic functions. Lebesgue and Sobolev spaces of time-periodic vector-valued functions are defined by
\begin{align*}
&\LRper{p}\left(\R; E(\Omega)\right):= \closure{\CRiper\left(\R; E(\Omega)\right)}{\norm{\cdot}_{\LR{p}\left(\torus; E(\Omega)\right)}},\\
&\WSRper{k}{p}\left(\R; E(\Omega)\right) := \closure{\CRiper\left(\R; E({\Omega}{})\right)}{\norm{\cdot}_{\WSR{k}{p}\left(\torus. E(\Omega)\right)}},
\end{align*}
where the norms $\norm{\cdot}_{\LR{p}\left(\torus; E(\Omega)\right)}$ and $\norm{\cdot}_{\WSR{k}{p}\left(\torus. E(\Omega)\right)}$ are defined as in \eqref{DefOfNormsOnTorus}.
Similarly, $\TPDper(\R\times\partial\Omega)$ and $\TPNper(\R\times\partial\Omega)$ are defined in accordance with 
\eqref{DefOfTraceSpaces}.

\begin{thm}[Dirichlet problem]\label{max_regD}
Assume that either $\Omega = \R^3$, $\Omega = \R^3_+$ or $\Omega\subset\R^3$ is a bounded domain with a $C^{1,1}$-smooth boundary. Let $p\in (1, \infty)$. Then for any $f\in\LRper{p}(\R; \LR{p}\left(\Omega\right))$ and $g\in \TPDper(\R\times\partial\Omega)$ there is a solution $\uvel$ to \eqref{damp_waveD} with
\begin{align}
\uvel(t,x) = \us(x) + \up(t,x) \in \WSRD{2}{p}\left(\Omega\right)\oplus\oPR\WSRper{2}{p}(\R; \LR{p}(\Omega)) \cap \oPR\WSRper{1}{p}\left(\R; \WSR{2}{p}(\Omega)\right)
\end{align}
satisfying
\begin{align}
&\snorm{\us}_{2,p}\leq c_1\left(\norm{\PR f}_p + \norm{\PR g}_{\TPD}\right)\label{est_1},\\
&\norm{\up}_{\SolS{p}}\leq c_2\left(\norm{\oPR f}_p + \norm{\oPR g}_{\TPD}\right)\label{est_2},
\end{align}
where $c_1 = c_1\left(p, \Omega\right)>0$ and $c_2 = c_2\left(p, \Omega, \per\right)>0$. If $v = v_s + v_p$ is another solution with $v_s\in\WSRD{2}{q_1}\left(\Omega\right)$ and $v_p\in \oPR\WSRper{2}{q_2}(\R; \LR{q_2}(\Omega)) \cap \oPR\WSRper{1}{q_2}\left(\R; \WSR{2}{q_2}(\Omega)\right), q_1, q_2\in (1, \infty)$, then $\us - v_s$ is a polynomial of order 1 and $\up = v_p$.
\end{thm}

\begin{thm}[Neumann problem]\label{max_regN}
Let $\Omega$ and $p$ be as in Theorem \ref{max_regD}. Furthermore, let $f\in\LRper{p}(\R; \LR{p}\left(\Omega\right))$ and $g\in\TPNper(\R\times\partial\Omega)$. If $\Omega$ is a bounded domain, assume
\begin{align}\label{CompCond}
\int_0^\per\int_\Omega f \,\dx\dt + \int_0^\per\int_{\partial\Omega} g \,\dS\dt = 0.
\end{align}
Then there is a solution $\uvel$ to \eqref{damp_waveN} with
\begin{align}
\uvel(t,x) = \us(x) + \up(t,x) \in \WSRD{2}{p}\left(\Omega\right)\oplus\oPR\WSRper{2}{p}(\R; \LR{p}(\Omega)) \cap \oPR\WSRper{1}{p}\left(\R; \WSR{2}{p}(\Omega)\right)
\end{align}
satisfying 
\begin{align}
&\snorm{\us}_{2,p}\leq c_1\left(\norm{\PR f}_p + \norm{\PR g}_{\TPN}\right)\label{est_3},\\
&\norm{\up}_{\SolS{p}}\leq c_2\left(\norm{\oPR f}_p + \norm{\oPR g}_{\TPN}\right)\label{est_4},
\end{align}
where $c_1 = c_1\left(p, \Omega\right)>0$ and $c_2 = c_2\left(p, \Omega, \per\right)>0$. If $v = v_s + v_p$ is another solution with $v_s\in\WSRD{2}{q_1}\left(\Omega\right)$ and $v_p\in \oPR\WSRper{2}{q_2}(\R; \LR{q_2}(\Omega)) \cap \oPR\WSRper{1}{q_2}\left(\R; \WSR{2}{q_2}(\Omega)\right), q_1, q_2\in (1, \infty)$, then $\us - v_s$ is a polynomial of order 1 and $\up = v_p$.
\end{thm}

\begin{cor}\label{homeo}
Let $\Omega$ and $p$ be as in Theorem \ref{max_regD}. The operator  
\begin{align*}
&\operatorname{A}\colon\oPR\WSRper{2}{p}(\R; \LR{p}(\Omega)) \cap \oPR\WSRper{1}{p}\left(\R; \WSR{2}{p}(\Omega)\right) \\ 
&\qquad\qquad\qquad\qquad\to \oPR\LRper{p}(\R; \LR{p}\left(\Omega\right))\times\oPR T^p_{S, per}(\R\times\partial\Omega), \\
&\operatorname{A}(\up) := \left(\partial_{t}^2\up - \Delta\up - \lambda\partial_t\Delta\up, \operatorname{Tr}_S\up\right), 
\end{align*}
with $S\in\{D, N\}$, is a homeomorphism.
\end{cor}
\begin{proof}
Follows directly from Theorem \ref{max_regD} and \ref{max_regN}.
\end{proof}

We divide the proof of Theorem \ref{max_regD} and Theorem \ref{max_regN} into a number of steps.

\subsection{The Whole-Space}\label{whole}

In the whole-space case the resolution to $\per$-time-periodic solutions to \eqref{damp_waveD} and \eqref{damp_waveN} is equivalent, via the quotient mapping $\pi$, to the resolution to the system 
\begin{align}\label{WDonGrp}
\partial_{t}^2\uvel - \Delta\uvel - \lambda\partial_t\Delta\uvel =  f \quad \tin\grp
\end{align}
on the group $\grp$. The system \eqref{WDonGrp} can be investigated in the framework introduced in Section \ref{GroupSettingSection}. 

We consider $f\in\LR{p}(\grp)$. We use the projections $\PR$ and $\oPR$ to decompose $f$ as $$f = \PR f + \oPR f \in\LR{p}(\R^3)\oplus\oPR\LR{p}(\grp),$$ and seek a solution $\uvel$ to \eqref{WDonGrp} as a sum 
$\uvel = \us + \up$,
where $\us$ is a solution to
\begin{align}\label{steady}
-\Delta\us = \PR f \quad \tin\R^3
\end{align}
and $\up$ a solution to
\begin{align}\label{damp_wave_per}
\partial_{t}^2\up - \Delta\up - \lambda\partial_t\Delta\up = \oPR f \quad \tin\grp.
\end{align}
The resolution to the elliptic problem \eqref{steady} is well known. We therefore turn our focus to \eqref{damp_wave_per}.

\begin{lem}\label{peri_whole}
Let $p\in (1,\infty)$. For any $f\in\oPR\LR{p}(\grp)$ there exists a solution ${\uvel\in \SolS{p}(\grp)}$ to \eqref{damp_wave_per}. Moreover 
\begin{align}\label{reg_per_whole}
\norm{\uvel}_{\SolS{p}}\leq c\norm{f}_p,
\end{align}
were $c=c(p, \per)>0$.
The solution $\uvel$ is unique in $\TDRper(\grp)$.
\end{lem}
\begin{proof}
Recall that $\FT_\grp\nb{\oPR f} = \left(1 - \delta_\Z\right)\FT_\grp\nb{f}$. Formally applying the Fourier transform $\FT_\grp$ in \eqref{damp_wave_per}, we therefore obtain
\begin{align}\label{sol_up}
\uvel = \iFT_\grp\left[\frac{\left(1 - \delta_\Z\right)}{|\xi|^2 - k^2 + i\lambda k|\xi|^2}\FT_\grp\nb{f}\right].
\end{align}
We put
\begin{align}
M: \dualgrp\ra\C, \quad M\left(k, \xi\right) := \frac{1 - \delta_\Z}{|\xi|^2 - k^2 + i\lambda k|\xi|^2}
\end{align}
and write
\begin{align}\label{sol_multi}
\uvel = \iFT_\grp\left[M\left(k, \xi\right)\FT_\grp\nb{f}\right].
\end{align}
Since $M\in\LR{\infty}(\dualgrp)$ is bounded, it is clear that $\uvel$ given by \eqref{sol_up} is well-defined as an element in $\TDR(\grp)$. We want to use the transference principle for multipliers, \textit{i.e.}, Lemma \ref{transference}, to establish \eqref{reg_per_whole}. For this purpose, let $\chi$ be a "cut-off" function with
\begin{align*}
\chi\in\CRci(\R; \R), \quad \chi(\eta) = 1 \text{ for } \left|\eta\right|\leq\frac{\pi}{\per}, \quad \chi(\eta) = 0 \text{ for }\left|\eta\right|\geq\frac{2\pi}{\per}.
\end{align*}
We then define
\begin{align}\label{multiplier_R}
m:\R\times\R^3\ra\C, \quad m\left(\eta, \xi\right) := \frac{1 - \chi(\eta)}{|\xi|^2 - \eta^2 + i\lambda\eta |\xi|^2}.
\end{align}
In order to employ Lemma \ref{transference}, we define the group $\grpH := \R\times\R^3$ and put
\begin{align}
\Phi: \dualgrp\ra\dualgrpH, \quad \Phi(k, \xi):= (k, \xi).
\end{align}
Recall that $\dualgrpH = \R\times\R^3$. Clearly, $\Phi$ is a continuous homomorphism. Moreover, 
\begin{align}\label{M}
M = m\circ\Phi.
\end{align}
Consequently, if we can show that $m$ is a continuous $\LR{p}(\grpH)$-multiplier, we may conclude from Lemma \ref{transference} that $M$ is an $\LR{p}(\grp)$-multiplier. Observe the only zero of the denominator in \eqref{multiplier_R} is $(\eta, \xi)=(0, 0)$. Since the numerator $1 - \chi(\eta)$ in \eqref{multiplier_R} vanishes in a neighborhood of $(0, 0)$, we see that $m$ is continuous; in fact $m$ is smooth. We shall now apply Marcinkiewicz's multiplier theorem to show that $m$ is an $\LR{p}(\grpH)$-multiplier. For this purpose we must verify that
\begin{align}\label{Mar_u}
\sup_{\varepsilon\in\{0, 1\}^4}\sup_{(\eta, \xi)\in\R\times\R^3}\left|\xi_1^{\varepsilon_1}\xi_2^{\varepsilon_2}\xi_3^{\varepsilon_3}\eta^{\varepsilon_4}\partial_{1}^{\varepsilon_1}\partial_{2}^{\varepsilon_2}\partial_{3}^{\varepsilon_3}\partial_{\eta}^{\varepsilon_4}m(\eta, \xi)\right|<\infty.
\end{align}
Since $m$ is smooth, we only need to show that all functions of type
\begin{align*}
\left(\eta, \xi\right)\ra \xi_1^{\varepsilon_1}\xi_2^{\varepsilon_2}\xi_3^{\varepsilon_3}\eta^{\varepsilon_4}\partial_{1}^{\varepsilon_1}\partial_{2}^{\varepsilon_2}\partial_{3}^{\varepsilon_3}\partial_{\eta}^{\varepsilon_4}m(\eta, \xi)
\end{align*}
stay bounded as $\left|\left(\eta, \xi\right)\right|\ra\infty$. Observe that these functions are rational functions with non-vanishing denominators away from $(0, 0)$.
Since $1 - \chi(\eta)$ vanish in a neighborhood of $(0,0)$, it follows that
$\left| m(\eta, \xi)\right| \leq c_0$.
We further estimate
\begin{align*}
|\eta\partial_\eta &m(\eta, \xi)| \leq \frac{\left|\chi^\prime(\eta)\right|\left|\eta\right|}{\sqrt{\left(|\xi|^2 - \eta^2\right)^2 + \lambda^2\eta^2|\xi|^4}} + \frac{\left| 1 - \chi(\eta)\right|\sqrt{4\eta^4 + \lambda^2\eta^2|\xi|^4}}{\left(|\xi|^2 - \eta^2\right)^2 + \lambda^2\eta^2|\xi|^4} \\
&\leq \frac{2\pi\left|\chi^\prime(\eta)\right|}{\per\sqrt{\left(|\xi|^2 - \eta^2\right)^2 + \lambda^2\eta^2|\xi|^4}} \\
&\quad + \frac{\left| 1 - \chi(\eta)\right|}{\sqrt{\left(|\xi|^2 - \eta^2\right)^2 + \lambda^2\eta^2|\xi|^4}}\sqrt{\frac{4}{\left(\frac{|\xi|^2}{\eta^2} - 1\right)^2 + \lambda^2\frac{|\xi|^4}{\eta^2}} + \frac{1}{\frac{\left(|\xi|^2 - \eta^2\right)^2}{\lambda^2\eta^2|\xi|^4} + 1}} \\
&\leq \frac{2\pi\left|\chi^\prime(\eta)\right|}{\per\sqrt{\left(|\xi|^2 - \eta^2\right)^2 + \lambda^2\eta^2|\xi|^4}} + c_0\sqrt{4c_1 + 1} 
\end{align*}
with $c_1 := \min(1, \frac{\per^2}{\lambda^2\pi^2})$. Since the denominator $\left(|\xi|^2 - \eta^2\right)^2 + \lambda^2\eta^2|\xi|^4$ does not vanish for $\eta\in\supp\left|\chi^\prime(\eta)\right|\subset\setc{\eta\in\R}{\frac{\pi}{\per}<\snorm{\eta}<\frac{2\pi}{\per}}$, there is a constant $c_2>0$ such 
\begin{align*}
|\eta\partial_\eta &m(\eta, \xi)| \leq c_2 + c_0\sqrt{4c_1 + 1}.
\end{align*}
For the partial derivative $\partial_{j} m$ we have
\begin{align*}
|\xi_j\partial_{j} m(\eta, \xi)| &= 2\left| 1 - \chi(\eta)\right| \frac{|\xi_j|^2 |1 + i\lambda\eta|}{\left||\xi|^2 - \eta^2 + i\lambda\eta|\xi|^2\right|^2} \\
&\leq \frac{2\left| 1 - \chi(\eta)\right|}{\sqrt{\left(|\xi|^2 - \eta^2\right)^2 + \lambda^2\eta^2|\xi|^4}} \sqrt{\frac{|\xi|^4 + \lambda^2\eta^2|\xi|^4}{\left(|\xi|^2 - \eta^2\right)^2 + \lambda^2\eta^2|\xi|^4}}\\
&\leq \frac{2\left| 1 - \chi(\eta)\right|}{\sqrt{\left(|\xi|^2 - \eta^2\right)^2 + \lambda^2\eta^2|\xi|^4}}\sqrt{\frac{1}{\lambda^2\eta^2} + 1} \leq 2c_0c_3,
\end{align*}
with $c_3 := \sqrt{\frac{\per^2}{\lambda^2\pi^2} + 1}$. Furthermore, 
\begin{align*}
|\xi_j\xi_k\partial_{j}\partial_{k} m(\eta, \xi)| \leq 8\left| 1 - \chi(\eta)\right|\frac{|\xi|^4\left( 1 + \lambda^2\eta^2\right)}{\left(\left(|\xi|^2 - \eta^2\right)^2 + \lambda^2\eta^2|\xi|^4\right)^{3/2}} \leq 8 c_0 c_3^2
\end{align*}
and
\begin{align*}
|\xi_j\eta\partial_{j}\partial_\eta m(\eta, \xi)| &\leq 2\left|\chi^\prime(\eta)\right|\frac{\sqrt{|\xi|^4 + \lambda^2\eta^2|\xi|^4}}{\left(|\xi|^2 - \eta^2\right)^2 + \lambda^2\eta^2|\xi|^4} + 2\left| 1 - \chi(\eta)\right|\frac{\lambda|\xi|^2 |\eta|}{\left(|\xi|^2 - \eta^2\right)^2 + \lambda^2\eta^2|\xi|^4} \\
&\quad + 4\left| 1 - \chi(\eta)\right|\frac{\sqrt{4\eta^4 + \lambda^2\eta^2|\xi|^4}\sqrt{|\xi|^4 + \lambda^2\eta^2|\xi|^4}}{\left(\left(|\xi|^2 - \eta^2\right)^2 + \lambda^2\eta^2|\xi|^4\right)^{3/2}}\\
&\leq \frac{2\left|\chi^\prime(\eta)\right| c_3}{\sqrt{\left(|\xi|^2 - \eta^2\right)^2 + \lambda^2\eta^2|\xi|^4}} + 2c_0 + 4c_0c_3\sqrt{4c_1 + 1} \\
&\leq \frac{\per}{\pi}c_2c_3 + 2c_0 + 4c_0c_3\sqrt{4c_1 + 1}.
\end{align*}
Boundedness of the terms with derivatives of third order is given by
\begin{align*}
|\xi_j\xi_k\eta\partial_{j}\partial_{k}\partial_\eta m(\eta, \xi)| &\leq \frac{8\left|\chi^\prime(\eta)\right|\cdot|\eta|\cdot|\xi|^4(1+\lambda^2\eta^2)}{\left(\left(|\xi|^2 - \eta^2\right)^2 + \lambda^2\eta^2|\xi|^4\right)^{3/2}} + \frac{16\left| 1 - \chi(\eta)\right|\cdot|\eta|\cdot|\xi|^4\sqrt{1+\lambda^2\eta^2}}{\left(\left(|\xi|^2 - \eta^2\right)^2 + \lambda^2\eta^2|\xi|^4\right)^{3/2}} \\
&\qquad + \frac{24\left| 1 - \chi(\eta)\right|\cdot|\xi|^4\left( 1 + \lambda^2\eta^2\right)\cdot|\eta|\sqrt{\lambda^2|\xi|^4 + 4\eta^2}}{\left(\left(|\xi|^2 - \eta^2\right)^2 + \lambda^2\eta^2|\xi|^4\right)^{2}} \\
&\leq \frac{8\left|\chi^\prime(\eta)\right| c_3^2}{\sqrt{\left(|\xi|^2 - \eta^2\right)^2 + \lambda^2\eta^2|\xi|^4}} + 16 c_0 c_3 + 24 c_0c_3^2\sqrt{4c_1 + 1} \\
&\leq \frac{4\per}{\pi}c_2c_3^2 + 16 c_0 c_3 + 24 c_0 c_3^2\sqrt{4c_1 + 1}
\end{align*}
and
\begin{align*}
|\xi_j\xi_k\xi_l\partial_{j}\partial_{k}\partial_l m(\eta, \xi)| &\leq 48\left| 1 - \chi(\eta)\right|\frac{|\xi|^6(1 + \lambda^2\eta^2)^{3/2}}{\left(\left(|\xi|^2 - \eta^2\right)^2 + \lambda^2\eta^2|\xi|^4\right)^{2}} \leq 48 c_0 c_3^3.
\end{align*}
We see that
\begin{align*}
|\xi_j\xi_k\xi_l\eta\partial_{j}&\partial_{k}\partial_l\partial_\eta m(\eta, \xi)| \leq \frac{192\left| 1 - \chi(\eta)\right|\cdot|\eta|\cdot|\xi|^6\left(1+\lambda^2\eta^2\right)^{3/2}\sqrt{\lambda^2|\xi|^4 + 4\eta^2}}{\left(\left(|\xi|^2 - \eta^2\right)^2 + \lambda^2\eta^2|\xi|^4\right)^{5/2}}\\
&+ \frac{144\left| 1 - \chi(\eta)\right|\cdot\lambda|\eta|\cdot|\xi|^6\left(1+\lambda^2\eta^2\right)}{\left(\left(|\xi|^2 - \eta^2\right)^2 + \lambda^2\eta^2|\xi|^4\right)^{2}} + \frac{48\left|\chi^\prime(\eta)\right|\cdot|\eta|\cdot|\xi|^6(1+\lambda^2\eta^2)^{3/2}}{\left(\left(|\xi|^2 - \eta^2\right)^2 + \lambda^2\eta^2|\xi|^4\right)^{2}} \\
&\leq 192 c_0c_3^3\sqrt{4c_1 + 1} + 144 c_0 c_3^2 + \frac{48\left|\chi^\prime(\eta)\right| c_3^3}{\sqrt{\left(|\xi|^2 - \eta^2\right)^2 + \lambda^2\eta^2|\xi|^4}} \\
&\leq 192 c_0c_3^3\sqrt{4c_1 + 1} + 144 c_0 c_3^2 + \frac{24\per}{\pi}c_2c_3^3.
\end{align*}
Consequently, we conclude \eqref{Mar_u} and by Marcinkiewicz's multiplier theorem that $m$ is an $\LR{p}(\grpH)$-multiplier. Hence, due to \eqref{M} it follows from Lemma \ref{transference} that $M$ is an $\LR{p}(\grp)$-multiplier. Recalling \eqref{sol_up} or \eqref{sol_multi}, we thus obtain 
\begin{align}\label{Lpu}
\norm{\uvel}_p\leq c\norm{f}_p.
\end{align}
Note that the neighborhood in which $m$ is vanishing becomes small as $\per\ra\infty$, and hence the corresponding bound in \eqref{Mar_u} grows for large periods $\per$. Differentiating $\uvel$ with respect to time and space, we obtain from \eqref{sol_up} the formulas
\begin{align*}
&\partial_t^\beta\uvel = \iFT_\grp\left[\left(ik\right)^{\beta} M\left(k, \xi\right)\FT_\grp\nb{f}\right]\\
&\partial_x^\alpha\uvel = \iFT_\grp\left[i^{|\alpha|}\xi^\alpha M\left(k, \xi\right)\FT_\grp\nb{f}\right]\\
&\partial_t\partial_x^\alpha\uvel = \iFT_\grp\left[i^{|\alpha| + 1}k\xi^\alpha M\left(k, \xi\right)\FT_\grp\nb{f}\right].
\end{align*}
We can repeat the argument above with $\left(ik\right)^{\beta} M(k, \xi)$ in the role of the multiplier $M$, and $(i\eta)^{\beta} m(\eta, \xi)$ in the role of $m$, to conclude
\begin{align}\label{Lput}
\norm{\partial_t^\beta\uvel}_{p}\leq c\norm{f}_p.
\end{align}
Similarly, we obtain
\begin{align}\label{Lpux}
\norm{\partial_x^\alpha\uvel}_{p}\leq c\norm{f}_p,\quad \norm{\partial_t\partial_x^\alpha\uvel}_{p}\leq c\norm{f}_p.
\end{align}
Collecting \eqref{Lpu}-\eqref{Lpux} we conclude \eqref{reg_per_whole}. Due to \eqref{sol_up} it is clear that $\oPR\uvel = \uvel$, whence we have $\uvel\in \SolS{p}(\grp)$.

It remains to show uniqueness. Assume that $v\in\TDR(\grp)$ is another solution with $\PR v = 0$. Therefore, we notice $$\partial_{t}^2\left(\uvel-v\right) - \Delta\left(\uvel-v\right) - \lambda\partial_t\Delta\left(\uvel-v\right) = 0.$$ Applying the Fourier transform $\FT_\grp$, it then follows ${\left(|\xi|^2 - k^2 + i\lambda k|\xi|^2\right)\FT_\grp\nb{\uvel - v} = 0}$ and thus $\supp\FT_\grp\nb{\uvel - v}\subset\{(0, 0)\}$. Recall that $\PR(\uvel - v)$ is time independent. From this we obtain that $\delta_\Z\cdot\FT_\grp\nb{\uvel - v} = \FT_\grp\nb{\PR(\uvel - v)} = 0$ and therefore we must have ${(0, 0)\notin\supp\FT_\grp\nb{\uvel - v}}$. Consequently, we conclude $\supp\FT_\grp\nb{\uvel - v} = \emptyset$ and ${\uvel = v}$.
\end{proof}

\subsection{Dirichlet Boundary Condition}\label{Dirichlet}
Next, we consider the damped wave equation with Dirichlet boundary conditions.
We first treat the half-space case, then the bent half-space case, and finally the bounded domain. 
We utilize the equivalence between the resolution to $\per$-time-periodic solutions to \eqref{damp_waveD} and the 
resolution of the system obtained by replacing the time axis in \eqref{damp_waveD} with the torus $\torus$. The latter system is investigated in
the framework introduced in Section \ref{FunktionSpacesSection}.

\subsubsection{The Half-Space}\label{half-space}

We first consider the half-space case
\begin{align}\label{damp_wave_Dirichlet}
\begin{pdeq}
\partial_{t}^2\uvel-\Delta\uvel - \lambda\partial_t\Delta\uvel &= f && \tin\torus\times\R_+^3, \\
\uvel &= g && \ton\torus\times\partial\R_+^3. 
\end{pdeq}
\end{align}
We make use of a reflection principle argument.
\begin{lem}\label{peri_Dirichlet}
Let $p\in\left(1, \infty\right)$. For any $f\in\oPR\LR{p}(\torus\times\R_+^3)$ and $g\in\oPR\TPD\left(\torus\times\partial\R_+^3\right)$ there exists a unique solution $\uvel\in \SolS{p}(\torus\times\R_+^3)$ to \eqref{damp_wave_Dirichlet} and there is a constant $c = c(p, \per) > 0$ such that
\begin{align}\label{reg_per_Dirichlet}
\norm{\uvel}_{\SolS{p}}\leq c\left(\norm{f}_p + \norm{g}_{\TPD}\right).
\end{align}
If additionally $f\in\oPR\LR{s}(\torus\times\R_+^3)$ and $g\in\oPR\T_D^s\left(\torus\times\partial\R_+^3\right)$ for some $s\in\left( 1, \infty\right)$, then also $\uvel\in \SolS{s}(\torus\times\R_+^3)$.
\end{lem}

\begin{proof}
For homogeneous boundary values, \textit{i.e.} $g = 0$, the existence of a solution $\uvel\in \SolS{p}(\torus\times\R_+^3)$ to \eqref{damp_wave_Dirichlet} satisfying \eqref{reg_per_Dirichlet} follows from the reflection principle in combination with Lemma \ref{peri_whole}. We demonstrate this principle for the Dirichlet problem. Define 
\begin{align*}
\tilde{f}(t, x) := \begin{pdeq} &f(t, x^\prime, x_3) && \text{ if } x_3\geq 0, \\
-&f(t, x^\prime, -x_3) && \text{ if } x_3<0,
\end{pdeq}
\end{align*}
with $x^\prime := (x_1, x_2)$. By Lemma \ref{peri_whole} there is a solution $\tilde{\uvel}\in \SolS{p}(\torus\times\R^3)$ to
\begin{align}\label{1}
\partial_{t}^2\tilde{\uvel}-\Delta\tilde{\uvel} - \lambda\partial_t\Delta\tilde{\uvel} &= \tilde{f} \quad\tin\torus\times\R^3
\end{align}
satisfying \eqref{reg_per_Dirichlet}. To classify $\uvel := \tilde{\uvel}_{|\torus\times\R_+^3}$ as a solution to \eqref{damp_wave_Dirichlet}, we still have to verify that $\uvel$ satisfies the boundary condition. For this purpose, we show that $v(t,x) := -\tilde{\uvel}(t,x^\prime, -x_3)$ is another solution to \eqref{1}. We observe that
\begin{align*}
\left(\partial_t^2 - \Delta - \lambda\partial_t\Delta\right) v(t, x) &= \left(-\partial_t^2 + \Delta + \lambda\partial_t\Delta\right) \tilde{\uvel}(t, x^\prime, -x_3) = -\tilde{f}(t, x^\prime, -x_3) = \tilde{f}(t, x).
\end{align*}
Since a solution to \eqref{1} is unique in the whole-space case by Lemma \ref{peri_whole}, we obtain $\tilde{\uvel}(t, x^\prime, x_3) = -\tilde{\uvel}(t, x^\prime, -x_3)$ and thus 
\begin{align*}
\TD\left[\tilde{\uvel}(t, x^\prime, x_3)_{|\torus\times\R^3_+}\right] = -\TD\left[\tilde{\uvel}(t, x^\prime, -x_3)_{|\torus\times\R^3_+}\right].
\end{align*}
Consequently, $\TD\left[\tilde{\uvel}_{|\torus\times\R^3_+}\right]=0$. 
We conclude that $\uvel(t, x) := \tilde{\uvel}(t, x)_{|\torus\times\R_+^3}$ is a solution to \eqref{damp_wave_Dirichlet} with $g = 0$
and satisfies \eqref{reg_per_Dirichlet}. Utilizing that  
$\TD:\SolS{p}(\torus\times\R^3_+)\ra\projcompl\T_D^p\left(\torus\times\partial\R_+^3\right)$ is continuous and surjective, we can extend this assertion to the case of 
inhomogeneous boundary values $g\in\oPR\TPD\left(\torus\times\partial\R_+^3\right)$ by a standard lifting argument.

Concerning uniqueness, let $\uvel\in \SolS{p}(\torus\times\R_+^3)$ be a solution to \eqref{damp_wave_Dirichlet} with data $f = 0$ and $g = 0$. Let $\psi\in\oPR\LR{p^\prime}(\torus\times\R_+^3)$ be arbitrary. By the argument above there is a $\phi\in \SolS{p^\prime}(\torus\times\R_+^3)$ such that $\partial_{t}^2\phi - \Delta\phi - \lambda\partial_t\Delta\phi = \psi$ and $\phi\big|_{\torus\times\partial\R_+^3} = 0$. Defining $\tilde{\phi}(t, x) := \phi(-t, x)$, we conclude
\begin{align*}
\frac{1}{\per}\int_0^\per\int_{\R^3_+} \uvel\psi \,\dx\dt &= \frac{1}{\per}\int_{0}^\per\int_{\R^3_+} \uvel\left(\partial_t^2\tilde{\phi} - \Delta\tilde{\phi} + \lambda\partial_t\Delta\tilde{\phi}\right)\,\dx\dt \\ 
&= \frac{1}{\per}\int_{0}^\per\int_{\R_+^3} \left(\partial_t^2\uvel - \Delta\uvel - \lambda\partial_t\Delta\uvel\right)\tilde{\phi} \,\dx\dt = 0.
\end{align*}
Since $\psi$ was arbitrary, it follows that $\uvel = 0$. 

Now assume in addition $f\in\oPR\LR{s}(\torus\times\R_+^3)$ for some $s\in\left( 1, \infty\right)$ and $g=0$. Using the reflection principle in the same way as above, we obtain a solution $\tilde{U}\in\SolS{s}(\torus\times\R^3)$. Lemma \ref{peri_whole} yields that $\tilde{U}$ is unique in $\TDR(\grp)$ and thus $\tilde{U} = \tilde{\uvel}$ in $\TDR(\grp)$. It follows that $\uvel\in\SolS{s}(\torus\times\R^3_+)$. 
By a standard lifting argument, same the is true for inhomogeneous boundary values $g\in\oPR\T_D^s\left(\torus\times\partial\R_+^3\right)$.
\end{proof}

\subsubsection{The Bent Half-Space}\label{Sec_Bent_D}

In the next step, we consider the Dirichlet problem in a bent half-space $\torus\times\R^3_{\omega}$. Here,
${\R^3_\omega := \setc{(x^\prime, x_3)\in\R^3}{x_3 > \omega(x^\prime)}}$ is a perturbation of the half-space $\R^3_+$ by a continuous function $\omega:\R^2\to\R$.

\begin{lem}\label{Bent_Dirichlet}
Let $p\in (1, \infty)$ and $\omega\in C^{0,1}(\R^2)$. There is a constant $\delta = \delta(p) >0$ with the following property: If $\norm{\grad\omega}_\infty, \norm{\grad^2\omega}_\infty < \delta$, then for any $f\in \oPR\LR{p}\left(\torus\times\R^3_{\omega}\right)$ and $g\in\oPR\TPD\left(\torus\times\partial\R^3_{\omega}\right)$ there exists a unique solution $\uvel\in \SolS{p}\left(\torus\times\R^3_{\omega}\right)$ to 
\begin{align}\label{bentD}
\begin{pdeq}
\partial_{t}^2\uvel-\Delta\uvel - \lambda\partial_t\Delta\uvel &= f && \tin\torus\times\R^3_{\omega}, \\
\uvel &= g && \ton\torus\times\partial\R^3_{\omega}.
\end{pdeq}
\end{align}
Moreover, there is a constant $c = c(p, \omega, \per) > 0$ such that
\begin{align}\label{est_Bent_D}
\norm{\uvel}_{\SolS{p}} \leq c\left(\norm{f}_p + \norm{g}_{\TPD}\right).
\end{align}
If additionally $\norm{\grad\omega}_\infty, \norm{\grad^2\omega}_\infty < \min\{\delta(p), \delta(s)\}$ and $f\in \oPR\LR{s}\left(\torus\times\R^3_{\omega}\right)$ and $g\in\oPR\T_D^s\left(\torus\times\partial\R^3_{\omega}\right)$ for some $s\in (1, \infty)$, then $\uvel\in \SolS{s}(\torus\times\R^3_{\omega})$.
\end{lem}

\begin{proof}
Let 
\begin{align}\label{Glatt}
\phi_\omega: \R^3_\omega\to\R^3_+,\quad \phi_\omega(x) := \tilde{x} := (x^\prime, x_3 - \omega(x^\prime)).
\end{align}
For a function $\uvel$ defined on $\torus\times\R^3_{\omega}$, we set $\Phi[\uvel](t, \tilde{x}) := \tilde{\uvel}(t, \tilde{x}) := \uvel(t, \phi^{-1}_\omega(\tilde{x}))$, where $(t, \tilde{x})\in\torus\times\R^3_+$. 
Observe that
\begin{align}\label{Glättung}
\Phi\bb{\left(\partial_t^2 - \Delta - \lambda\partial_t\Delta\right)\uvel} = \left(\partial_t^2 - \Delta - \lambda\partial_t\Delta + \tilde{R}\right)\tilde{\uvel},
\end{align}
where $\tilde{R}: \SolS{p}\left(\torus\times\R^3_+\right) \to \oPR\LR{p}\left(\torus\times\R^3_+\right)$ is given by 
\begin{align}\label{Rtilde}
\begin{aligned}
\tilde{R}\tilde{\uvel} := &-\left|\grad\omega\right|^2\partial_3^2\tilde{\uvel} + 2\left(\grad\omega, 0\right)\grad\partial_3\tilde{\uvel} + \left(\Delta\omega\right)\partial_3\tilde{\uvel} \\
&- \lambda\left|\grad\omega\right|^2\partial_t\partial_3^2\tilde{\uvel} + 2\lambda\left(\grad\omega, 0\right)\grad\partial_t\partial_3\tilde{\uvel} + \lambda\left(\Delta\omega\right)\partial_t\partial_3\tilde{\uvel}.
\end{aligned}
\end{align}
Moreover, due to \eqref{reg_per_Dirichlet} we can estimate 
\begin{align}\label{Bent_DirichletPerturbation}
\begin{aligned}
\norm{\tilde{R}\tilde{\uvel}}_p &\leq 8\delta\left(\delta+1\right)\norm{\tilde{\uvel}}_{\SolS{p}} \\
&\leq c\,8\delta\left(\delta+1\right)\bp{\norm{\left(\partial_t^2 - \Delta - \lambda\partial_t\Delta\right)\tilde{\uvel}}_{p}+\norm{\TD\tilde{\uvel}}_{\T_D^p}}.
\end{aligned}
\end{align}
By Lemma \ref{peri_Dirichlet}, the operator 
\begin{align*}
&\tilde{\calk}:\SolS{p}(\torus\times\R^3_+)\ra\oPR\LR{p}(\torus\times\R^3_+)\times\oPR\T_D^p(\torus\times\partial\R^3_+),\\
&\tilde{\calk}(\tilde{\uvel}):= \bp{\partial_t^2\tilde{\uvel} - \Delta\tilde{\uvel} - \lambda\partial_t\Delta\tilde{\uvel},\TD\tilde{\uvel}}
\end{align*}
is a homeomorphism. For sufficiently small $\delta$, we infer from \eqref{Bent_DirichletPerturbation} that also 
\begin{align*}
&\overline{\calk}:\SolS{p}(\torus\times\R^3_+)\ra\oPR\LR{p}(\torus\times\R^3_+)\times\oPR\T_D^p(\torus\times\partial\R^3_+),\\
&\overline{\calk}(\tilde{\uvel}):= \bp{\partial_t^2\tilde{\uvel} - \Delta\tilde{\uvel} - \lambda\partial_t\Delta\tilde{\uvel}+\tilde{R}\tilde{\uvel},\TD\tilde{\uvel}}
\end{align*}
is a homeomorphism. Since $\norm{\grad\omega}_\infty, \ \norm{\grad^2\omega}_\infty < \infty$, it is standard to verify that
\begin{align*}
&\Phi: \oPR\LR{p}\left(\torus\times\R^3_{\omega}\right)\to \oPR\LR{p}\left(\torus\times\R^3_{+}\right),\\
&\Phi: \SolS{p}\left(\torus\times\R^3_{\omega}\right)\to \SolS{p}\left(\torus\times\R^3_{+}\right),\\
&\Phi: \projcompl\T_D^{p}\left(\torus\times\partial\R^3_{\omega}\right)\to \projcompl\T_D^{p}\left(\torus\times\partial\R^3_{+}\right)
\end{align*}
are homeomorphisms.
From \eqref{Glättung} we thus deduce that  
\begin{align*}
&{\calk}:\SolS{p}(\torus\times\R^3_{\omega})\ra\oPR\LR{p}(\torus\times\R^3_\omega)\times\oPR\T_D^p(\torus\times\partial\R^3_{\omega}),\\
&{\calk}({\uvel}):= \bp{\partial_t^2{\uvel} - \Delta{\uvel} - \lambda\partial_t\Delta{\uvel},\TD{\uvel}}
\end{align*}
is a homeomorphism. The existence of a unique solution to \eqref{bentD} that satisfies \eqref{est_Bent_D} thus follows. The regularity assertion follows if we consider intersection spaces $\SolS{p}\cap\SolS{s}$ instead of $\SolS{p}$ in the argument above.
\end{proof}

\subsubsection{Bounded Domains}

The key lemma for bounded domains $\Omega\subset\R^3$ with a boundary of class $C^{1,1}$ reads as follows.

\begin{lem}\label{inj_dense_D}
Let $\Omega\subset\R^3$ be a bounded domain with boundary of class $C^{1,1}$ and let $p\in (1, \infty)$. The operator 
\begin{align*}
&\calk:\SolS{p}(\torus\times\Omega)\ra\oPR\LR{p}(\torus\times\Omega)\times\oPR\T_D^p(\torus\times\partial\Omega),\\
&\calk(\uvel):= \bp{\partial_t^2\uvel - \Delta\uvel - \lambda\partial_t\Delta\uvel,\TD\uvel}
\end{align*}
is injective and has a dense range. Moreover, there exists a constant $c = c(p, \Omega, \per) > 0$ such that for all $\uvel\in \SolS{p}(\torus\times\Omega)$ holds
\begin{align}\label{est_BD_D}
\norm{\uvel}_{\SolS{p}} \leq c\left(\norm{\left(\partial_t^2 - \Delta - \lambda\partial_t\Delta\right)\uvel}_p + \norm{\uvel}_p + \norm{\TD\uvel}_{\TPD}\right).
\end{align}
\end{lem}

\begin{proof}
Consider for $k\in\frac{2\pi}{\per}\Z\setminus\{0\}$ the equation 
\begin{align}\label{Helmholtz_Dirichlet}
\begin{pdeq}
-k^2v-\left( 1 + ik\lambda\right)\Delta v &= h && \tin\Omega, \\
v &= 0 && \ton\partial\Omega. 
\end{pdeq}
\end{align}
Standard elliptic theory yields for every $h\in\LR{p}(\Omega)$ a unique solution $v\in\WSR{2}{p}(\Omega)$ to \eqref{Helmholtz_Dirichlet}. If $\uvel\in \SolS{p}(\torus\times\Omega)$ satisfies $\calk(\uvel)= 0$, then $\FT_\torus\nb{u}\left(k, \cdot\right)\in\WSR{2}{p}\left(\Omega\right)$ solves \eqref{Helmholtz_Dirichlet} with a homogeneous right-hand side. Here $\FT_\torus$ denotes the Fourier transform on the torus. Consequently $\FT_\torus\nb{u}\left(k, \cdot\right) = 0$. Since $k\in\frac{2\pi}{\per}\Z\setminus\{0\}$ was arbitrary and $\FT_\torus\nb{u}\left(0, \cdot\right) = 0$ by the assumption $\PR\uvel = 0$, it follows that $\uvel = 0$. Consequently, $\calk$ is injective.

To show that $\calk$ has a dense range, consider  $(f,g)\in\oPR\LR{p}(\torus\times\Omega)\times\oPR\T_D^p(\torus\times\partial\Omega)$.
Choose $G\in\SolS{p}(\torus\times\Omega)$ with $\TD G=g$. Since trigonometric polynomials are dense in 
$\LR{p}\bp{\torus;\LR{p}(\Omega)}=\LR{p}(\torus\times\Omega)$, there 
is a sequence $\seqN{p_n}\subset \LR{p}(\torus\times\Omega)$ of trigonometric polynomials with $p_n\ra f-\left(\partial_t^2 G - \Delta G - \lambda\partial_t\Delta G\right)$. If we can find a solution $\tilde\uvel_n$ to $\calk(\tilde\uvel_n)=(p_n,0)$, then $\calk(\tilde\uvel_n+G)\ra(f,g)$, and 
we may conclude density of $\calk$'s range. To show existence of $\tilde\uvel_n$, it clearly suffices to solve $\calk(\tilde\uvel_n)=(p_n,0)$
for a simple trigonometric polynomial $p_n:=h\e^{ikt}$ with arbitrary $h\in\LR{q}(\Omega)$ and $k\in\frac{2\pi}{\per}\Z\setminus\{0\}$. A solution to this problem is given by $\tilde\uvel_n := v_k \e^{ikt}$, where $v_k$ is the solution to \eqref{Helmholtz_Dirichlet}.

Finally, we show \eqref{est_BD_D} by a localization method. We choose finitely many balls $B_j\subset\R^3$, $j\in\{1, \ \ldots, \ m\}$ covering $\Omega$, where each $j\in\{1, \ \ldots, \ m\}$ is of one of the two types:
\begin{enumerate}
\item type $\R^3$: if $\closure{B}{}_j\subset\Omega$,
\item type $\R_{\omega_j}^3$: if $\closure{B}{}_j\cap\partial\Omega\neq\emptyset$.
\end{enumerate}
In the second case, $\omega_j\colon\R^{2}\to\R$ denote Lipschitz functions with 
${\closure{B}{}_j\cap\partial\Omega \subset \graph(\omega_j)}$ in the respective local coordinates. If we choose the balls sufficiently small, the functions $\omega_j$ meet the regularity and smallness assumption in Lemma \ref{Bent_Dirichlet} due to the boundary regularity of $\Omega$. Let $\psi_j\in\CRci\left(\R^3\right)$ be smooth cut--off functions satisfying $\supp\psi_j\subset B_j$ and $\sum\limits_{j=1}^m \psi_j = 1$ in $\Omega$. We obtain for $j\in\{1, \ \ldots, \ m\}$
\begin{align}\label{Ungl_Kugel}
\partial_t^2\left(\psi_j\uvel\right) - \Delta\left(\psi_j\uvel\right) - \lambda\partial_t\Delta\left(\psi_j\uvel\right) = f_j \quad \tin \torus\times\Omega\cap B_j,
\end{align}
where 
\begin{align*}
f_j := \psi_j \left(\partial_t^2 - \Delta - \lambda\partial_t\Delta\right)\uvel - \left(\Delta\psi_j\right)\uvel - 2\left(\nabla\psi_j\right)\nabla\uvel - \lambda\left(\Delta\psi_j\right)\partial_t\uvel - 2\lambda\left(\nabla\psi_j\right)\partial_t\nabla\uvel.
\end{align*}
Depending on whether $j\in\{1, \ \ldots, \ m\}$ is of type $\R^3$ or $\R^3_{\omega_j}$, we interpret \eqref{Ungl_Kugel} as a problem in $\torus\times\R^3$ or $\torus\times\R^3_{\omega_j}$ and obtain from Lemma \ref{peri_whole} or Lemma \ref{Bent_Dirichlet} 
\begin{align*}
\norm{\psi_j\uvel}_{\SolS{p}} 
&\leq c\big(\norm{\left(\partial_t^2 - \Delta - \lambda\partial_t\Delta\right)\uvel}_p + \norm{\uvel}_p + \norm{\grad\uvel}_p \\ 
&\qquad +\norm{\partial_t\uvel}_p + \norm{\partial_t\grad\uvel}_p + \norm{\TD\uvel}_{\TPD}\big).
\end{align*}
Summing up over $j\in\{1, \ \ldots, \ m\}$ and using standard interpolation, \eqref{est_BD_D} follows.
\end{proof}

The next step is to show that we can drop the term $\norm{\uvel}_p$ on the right-hand side in \eqref{est_BD_D}.

\begin{lem}\label{closed_range_D}
Let $p\in (1, \infty)$ and $\Omega\subset\R^3$ be a bounded domain with boundary of class $C^{1,1}$. There exists a constant $c = c\left(p, \Omega, \per\right) >0$ such that for all $\uvel\in \SolS{p}\left(\torus\times\Omega\right)$ holds
\begin{align}\label{Absch_D}
\norm{\uvel}_{\SolS{p}}\leq c\left(\norm{\left(\partial_t^2 - \Delta - \lambda\partial_t\Delta\right)\uvel}_p + \norm{\TD\uvel}_{\TPD}\right).
\end{align}
\end{lem}
\begin{proof}
If \eqref{Absch_D} does not hold, then we find a sequence $\left(\uvel_k\right)_{k\in\N}\subset \SolS{p}\left(\torus\times\Omega\right)$ such that $\norm{\uvel_k}_{\SolS{p}} = 1$ for all $k\in\N$ and $\norm{\left(\partial_t^2 - \Delta - \lambda\partial_t\Delta\right)\uvel_k}_{p} + \norm{\TD\uvel_k}_{\TPD}\to 0$ as $k\to\infty$. Suppressing the notation of subsequences, we thus have the weak convergence $\uvel_k\rightharpoonup\uvel$ in $\SolS{p}\left(\torus\times\Omega\right)$, and $\uvel$ solves
\begin{align*}
\begin{pdeq}
\partial_t^2\uvel - \Delta\uvel - \lambda\partial_t\Delta\uvel &= 0 && \tin\torus\times\Omega, \\
\uvel &= 0 && \ton\torus\times\partial\Omega.
\end{pdeq}
\end{align*}
By Lemma \ref{inj_dense_D} it follows that $\uvel = 0$. Since the domain $\Omega$ is bounded, the embedding $\SolS{p}\left(\torus\times\Omega\right)\hookrightarrow\LR{p}\left(\torus\times\Omega\right)$ is compact, whence $\norm{\uvel_k}_p\rightarrow 0$ as $k\to\infty$. This yields the contradiction
\begin{align*}
1 = \lim\limits_{k\to\infty}\norm{\uvel_k}_{\SolS{p}}\leq\lim\limits_{k\to\infty}c\left(\norm{\left(\partial_t^2 - \Delta - \lambda\partial_t\Delta\right)\uvel_k}_p + \norm{\uvel_k}_p + \norm{\TD\uvel_k}_{\TPD}\right) = 0.
\end{align*}
Therefore, \eqref{Absch_D} has to hold.
\end{proof}
\begin{lem}\label{BD_reg_D}
Let $p\in (1, \infty)$ and $\Omega\subset\R^3$ be a bounded domain of class $C^{1,1}$. For any $f\in\oPR\LR{p}\left(\torus\times\Omega\right)$ and $g\in\oPR\TPD\left(\torus\times\partial\Omega\right)$ there exists a unique solution $\uvel\in \SolS{p}\left(\torus\times\Omega\right)$ to
\begin{align}\label{Wave_BD_Dirichlet}
\begin{pdeq}
\partial_t^2\uvel - \Delta\uvel - \lambda\partial_t\Delta\uvel &= f && \tin\torus\times\Omega, \\
\uvel &= g && \ton\torus\times\partial\Omega,
\end{pdeq}
\end{align}
and there is a constant $c = c(p, \Omega, \per) > 0$ such that
\begin{align}\label{est_BD_Dirichlet}
\norm{\uvel}_{\SolS{p}}\leq c\left(\norm{f}_p + \norm{g}_{\TPD}\right).
\end{align}
If additionally $f\in\oPR\LR{s}\left(\torus\times\Omega\right)$ and $g\in\oPR\T_D^s\left(\torus\times\partial\Omega\right)$ for some $s\in (1, \infty)$, then also $\uvel\in \SolS{s}\left(\torus\times\Omega\right)$.
\end{lem}
\begin{proof}
The operator $\calk$ in Lemma \ref{inj_dense_D} is injective and has a dense range. By Lemma \ref{closed_range_D}, the range is also closed. Hence, $\calk$ is an isomorphism. The unique solvability of \eqref{Wave_BD_Dirichlet} as well as the estimate \eqref{est_BD_Dirichlet} follows. The regularity assertion follows immediately from the unique solvability of \eqref{Wave_BD_Dirichlet} in $\SolS{\min\{s, p\}}\left(\torus\times\Omega\right)$.
\end{proof}

\subsection{Neumann Boundary Condition}\label{Neumann}

We now consider the corresponding Neumann problems in a half-space and a bounded domain.

\subsubsection{The Half-Space}
We first consider the half-space case 
\begin{align}\label{damp_wave_Neumann}
\begin{pdeq}
\partial_{t}^2\uvel-\Delta\uvel - \lambda\partial_t\Delta\uvel &= f && \tin\torus\times\R_+^3, \\
\pdn\uvel &= g && \ton\torus\times\partial\R_+^3. 
\end{pdeq}
\end{align}

\begin{lem}\label{peri_Neumann}
Let $p\in\left(1, \infty\right)$. For any $f\in\oPR\LR{p}\left(\torus\times\R_+^3\right)$ and $g\in\oPR\TPN\left(\torus\times\partial\R_+^3\right)$ there exists a unique solution $\uvel\in \SolS{p}\left(\torus\times\R_+^3\right)$ to \eqref{damp_wave_Neumann} and there is a constant $c = c(p, \per) > 0$ such that
\begin{align}\label{reg_per_Neumann}
\norm{\uvel}_{\SolS{p}}\leq c\left(\norm{f}_p + \norm{g}_{\TPN}\right).
\end{align}
If additionally $f\in\oPR\LR{s}\left(\torus\times\R_+^3\right)$ and $g\in\oPR T^s_N\left(\torus\times\partial\R_+^3\right)$ for some $s\in\left( 1, \infty\right)$, then also $\uvel\in\SolS{s}\left(\torus\times\R_+^3\right)$.
\end{lem}

\begin{proof}
Existence of a solution ${{\uvel}\in \SolS{p}\left(\torus\times\R_+^3\right)}$ to \eqref{damp_wave_Neumann} satisfying \eqref{reg_per_Neumann} follows as in the case of Dirichlet boundary values by using even instead of odd reflection in combination with Lemma \ref{peri_whole}.
Uniqueness of the solution in the space $\SolS{p}\left(\torus\times\R_+^3\right)$ follows as in Lemma \ref{peri_Dirichlet}. 
\end{proof}

\subsubsection{The Bent Half-Space}

Next, we study the Neumann problem in the bent half-space $\torus\times\R^3_\omega$. Here, $\omega$ is defined as in section \ref{Sec_Bent_D}.

\begin{lem}\label{Bent_Neumann}
Let $p\in (1, \infty)$ and $\omega\in C^{1,1}(\R^2)$. Then there is a constant $\delta = \delta(p) >0$ with the following property: If $\norm{\grad\omega}_\infty, \norm{\grad^2\omega}_\infty < \delta$, then for any ${f\in\oPR\LR{p}\left(\torus\times\R^3_\omega\right)}$ and $g\in\oPR\TPN\left(\torus\times\partial\R^3_\omega\right)$ there exists a unique solution ${\uvel\in\SolS{p}\left(\torus\times\R^3_\omega\right)}$ to
\begin{align}\label{Bent_N}
\begin{pdeq}
\partial_{t}^2\uvel - \Delta\uvel - \lambda\partial_t\Delta\uvel &= f && \tin\torus\times\R^3_\omega, \\
\pdn\uvel &= g && \ton\torus\times\partial\R^3_\omega,
\end{pdeq}
\end{align}
which satisfies
\begin{align}\label{est_Bent_N}
\norm{\uvel}_{\SolS{p}}\leq c\left(\norm{f}_p + \norm{g}_{\TPN}\right),
\end{align}
where $c = c\left(p, \omega, \per\right) > 0$. 
If additionally $\norm{\grad\omega}_\infty, \norm{\grad^2\omega}_\infty < \min\{\delta(p), \delta(s)\}$ and $f\in \oPR\LR{s}\left(\torus\times\R^3_\omega\right)$ and $g\in\oPR T^s_N\left(\torus\times\partial\R^3_\omega\right)$ for some $s\in (1, \infty)$, then $\uvel\in \SolS{s}(\torus\times\R^3_\omega)$.
\end{lem}

\begin{proof}
Let $\phi_\omega$ be as in \eqref{Glatt} and $\Phi$ be the lifting operator $\Phi[\uvel](t, \tilde{x}) := \tilde{\uvel}(t, \tilde{x}) := \uvel(t, \phi_\omega^{-1}(\tilde{x}))$. Then $\Phi$ is a homeomorphism $\Phi\colon\projcompl\TPN\left(\torus\times\partial\R^3_{\omega}\right)\to \projcompl\TPN\left(\torus\times\partial\R^3_{+}\right)$ with 
\begin{align}
\begin{aligned}
\Phi\left[\TN\uvel\right] &= \left(\grad\uvel\circ\phi_\omega^{-1}\cdot n\circ\phi_\omega^{-1}\right)\\
&= \left(\grad\left(\uvel\circ\phi_\omega^{-1}\right)\grad\phi_\omega \snorm{\operatorname{cof}\grad\phi_\omega \cdot n\circ\phi_\omega^{-1}}\left(\operatorname{cof}\grad\phi_\omega\right)^{-1} \tilde{n}\right) \\
&= \snorm{\operatorname{cof}\grad\phi_\omega \cdot n\circ\phi_\omega^{-1}} \left(\grad\tilde{\uvel}\grad\phi_\omega\left(\grad\phi_\omega\right)^\top\tilde{n} \right) \\
&= \snorm{\operatorname{cof}\grad\phi_\omega \cdot n\circ\phi_\omega^{-1}} \left( \grad\tilde{\uvel}\cdot\tilde{n} + \grad\tilde{\uvel}\left(\grad\phi_\omega\left(\grad\phi_\omega\right)^\top - I\right)\tilde{n}\right) \\
&= \snorm{\operatorname{cof}\grad\phi_\omega \cdot n\circ\phi_\omega^{-1}}\TN\tilde{\uvel} + \TD\tilde{S}\tilde{\uvel},
\end{aligned}
\end{align}
with
\begin{align*}
\tilde{S}\tilde{\uvel} := \snorm{\operatorname{cof}\grad\phi_\omega \cdot n\circ\phi_\omega^{-1}}\grad\tilde{\uvel}\left(\grad\phi_\omega\left(\grad\phi_\omega\right)^\top - I\right)\tilde{n}.
\end{align*}
Here, $\tilde{n}$ denotes the external unit normal vector on $\torus\times\R^3_+$ and $n$ the external unit normal vector on $\torus\times\R^3_\omega$. It should be understood that $\TN\uvel$ denotes the Neumann trace operator in $\torus\times\R^3_\omega$ and $\TN\tilde{\uvel}$ the Neumann trace operator in $\torus\times\R^3_+$. Due to \eqref{reg_per_Neumann}, we can estimate
\begin{align}\label{Bent_NeumannPerturbation}
\begin{aligned}
\norm{\tilde{S}\tilde{\uvel}}_{\TPN} &\leq c\norm{\operatorname{cof}\grad\phi_\omega \cdot n\circ\phi_\omega^{-1}}_{\WSR{1}{\infty}\left(\partial\R_+^3\right)}\norm{\grad\tilde{\uvel}\left(\grad\phi_\omega\left(\grad\phi_\omega\right)^\top - I\right)\tilde{n}}_{\TPN} \\
&\leq c(1+\delta)\norm{\grad\tilde{\uvel}}_{\TPN}\cdot\norm{\left(\grad\phi_\omega\left(\grad\phi_\omega\right)^\top - I\right)\tilde{n}}_{\WSR{1}{\infty}\left(\partial\R_+^3\right)} \\
&\leq c\delta\left(1+\delta\right)^2\norm{\tilde{\uvel}}_{\SolS{p}} \leq c\delta\left(1+\delta\right)^2\bp{\norm{\left(\partial_t^2 - \Delta - \lambda\partial_t\Delta\right)\tilde{\uvel}}_{p}+\norm{\TN\tilde{\uvel}}_{\T_N^p}}.
\end{aligned}
\end{align}
Lemma \ref{peri_Neumann} implies that
\begin{align*}
&\calk_+\colon\SolS{p}\left(\torus\times\R_+^3\right)\ra\oPR\LR{p}\left(\torus\times\R_+^3\right)\times\oPR\TPN\left(\torus\times\partial\R^3_+\right), \\
&\calk_+(\tilde{\uvel}) := \left(\partial_t^2\tilde{\uvel} - \Delta\tilde{\uvel} - \lambda\partial_t\Delta\tilde{\uvel}, \TN\tilde{\uvel}\right)
\end{align*}
is a homeomorphism. The operator 
\begin{align*}
&\calk\colon\SolS{p}\left(\torus\times\R_+^3\right)\ra\oPR\LR{p}\left(\torus\times\R_+^3\right)\times\oPR\TPN\left(\torus\times\partial\R^3_+\right), \\
&\calk(\tilde{\uvel}) := \calk_+(\tilde{\uvel}) + (\tilde{R}\tilde{\uvel}, -\tilde{S}\tilde{\uvel}),
\end{align*}
with $\tilde{R}$ defined as in \eqref{Rtilde}, is a perturbation of $\calk_+$ for sufficiently small $\delta$
by \eqref{Bent_DirichletPerturbation} and \eqref{Bent_NeumannPerturbation}. Consequently, 
$\calk$ is also a homeomorphism sufficiently small $\delta$.

By a direct computation, we observe that the operator on the left-hand side of \eqref{Bent_N} can be expressed as $\Phi^{-1}\circ\calk\circ\Phi$, and thus it is a homeomorphism. Hence, the existence of a solution $\uvel\in\SolS{p}\left(\torus\times\R^3_\omega\right)$ satisfying \eqref{est_Bent_N} follows. 

Now assume in addition $f\in \oPR\LR{s}\left(\torus\times\R^3_\omega\right)$ and $g\in\oPR T^s_N\left(\torus\times\partial\R^3_\omega\right)$ for some $s\in (1, \infty)$. We obtain $\uvel\in\SolS{s}\left(\torus\times\R^3_\omega\right)$ if we consider intersection spaces $\SolS{p}\cap\SolS{s}$ instead of $\SolS{p}$ in the argument above.
\end{proof}

\subsubsection{Bounded Domains}
Finally, we investigate the Neumann problem for a bounded domain. 

\begin{lem}\label{inj_dense_N}
Let $\Omega\subset\R^3$ be a bounded domain with boundary of class $C^{1,1}$ and let $p\in (1, \infty)$. The operator 
\begin{align*}
&\calk:\SolS{p}(\torus\times\Omega)\ra\oPR\LR{p}(\torus\times\Omega)\times\oPR\T_N^p(\torus\times\partial\Omega),\\
&\calk(\uvel):= \bp{\partial_t^2\uvel - \Delta\uvel - \lambda\partial_t\Delta\uvel,\TN\uvel}
\end{align*}
is injective and has a dense range. Moreover, there exists a constant $c = c\left(p, \Omega, \per\right) > 0$ such that for all $\uvel\in\SolS{p}\left(\torus\times\Omega\right)$ holds
\begin{align}\label{est_BD_N}
\norm{\uvel}_{\SolS{p}}\leq c\left(\norm{\left(\partial_t^2 - \Delta - \lambda\partial_t\Delta\right)\uvel}_p + \norm{\uvel}_p + \norm{\TN\uvel}_{\TPN}\right).
\end{align}
\end{lem}

\begin{proof}
Like in the case of Dirichlet boundary value problem, we consider for $k\in\frac{2\pi}{\per}\Z\setminus\{0\}$ the Helmholtz equation
\begin{align}\label{Helmholtz_Neumann}
\begin{pdeq}
-k^2v-\left( 1 + ik\lambda\right)\Delta v &= h && \tin\Omega, \\
\pdn v &= 0 && \ton\partial\Omega. 
\end{pdeq}
\end{align}
Standard theory for elliptic equations yields for every $h\in\LR{p}\left(\Omega\right)$ a unique solution $v\in\WSR{2}{p}\left(\Omega\right)$ to \eqref{Helmholtz_Neumann}. The injectivity of $\calk$ and the density of its range follow as in the proof 
of Lemma \ref{inj_dense_D}. Proceeding as in the proof of Lemma \ref{inj_dense_D}, \eqref{est_BD_N} follows.
\end{proof}

\begin{lem}\label{BD_reg_N}
Let $\Omega\subset\R^3$ be a bounded domain with boundary of class $C^{1,1}$ and let $p\in (1,\infty)$. For any $f\in\oPR\LR{p}\left(\torus\times\Omega\right)$ and $g\in\oPR\TPN\left(\torus\times\partial\Omega\right)$ there exists a unique solution $\uvel\in\SolS{p}\left(\torus\times\Omega\right)$ to
\begin{align}\label{Wave_BD_Neumann}
\begin{pdeq}
\partial_{t}^2\uvel - \Delta\uvel - \lambda\partial_t\Delta\uvel &= f && \tin\torus\times\Omega, \\
\pdn\uvel &= g && \ton\torus\times\partial\Omega,
\end{pdeq}
\end{align}
and there is a constant $c = c\left(p, \Omega, \per\right)$ such that the estimate
\begin{align}\label{est_BD_Neumann}
\norm{\uvel}_{\SolS{p}}\leq c\left(\norm{f}_p + \norm{g}_{\TPN}\right)
\end{align}
holds. If additionally $f\in\oPR\LR{s}\left(\torus\times\Omega\right)$ and $g\in\oPR T^s_N\left(\torus\times\partial\Omega\right)$ for some $s\in (1, \infty)$, then also $\uvel\in\SolS{s}\left(\torus\times\Omega\right)$.
\end{lem}

\begin{proof}
The operator $\calk$ is injectiv and has a dense range by Lemma \ref{inj_dense_N}. As in Lemma \ref{closed_range_D}, we can omit the mid term on the right-hand side in \eqref{est_BD_N} and obtain 
\begin{align}\label{Ungl_3}
\norm{\uvel}_{\SolS{p}}\leq c\left(\norm{\left(\partial_t^2 - \Delta - \lambda\partial_t\Delta\right)\uvel}_p + \norm{\TN\uvel}_{\TPN}\right).
\end{align}
It follows that the range of $\calk$ is also closed. Hence, $\calk$ is a homeomorphism. The unique solvability of \eqref{Wave_BD_Neumann} as well as \eqref{est_BD_Neumann} follows. The regularity assertion follows immediately from the unique solvability of \eqref{Wave_BD_Neumann} in $\SolS{\min\{s, p\}}\left(\torus\times\Omega\right)$.
\end{proof}

\subsection{Proof of the Theorems \ref{max_regD} and \ref{max_regN}}
\begin{proof}[Proof of Theorem \ref{max_regD}]
Existence of a solution $\us\in\WSRD{2}{p}\left(\Omega\right)$ to 
\begin{align}\label{Steady}
\begin{pdeq}
-\Delta\us &= \PR f && \tin\Omega, \\
\us &= \PR g && \ton\partial\Omega
\end{pdeq}
\end{align}
that satisfies \eqref{est_1} is well-known from standard theory on elliptic equations. Via the canonical quotient map $\pi\colon\R\to\torus$, the spaces $\CRiper\left(\R; E(\Omega)\right)$ and $\CRi\left(\torus; E(\Omega)\right)$ are isometrically isomorphic in the norms $\norm{\cdot}_p$ and $\norm{\cdot}_{\SolS{p}}$ for any Banach space $E$. By construction, also the Sobolev spaces $\WSRper{k}{p}\left(\R; E\left(\Omega\right)\right)$ and $\WSR{k}{p}\left(\torus; E\left(\Omega\right)\right)$ are isometrically isomorphic for any Banach space $E$. Hence Lemma \ref{whole} in the case $\Omega = \R^3$, Lemma \ref{peri_Dirichlet} in the case $\Omega = \R_+^3$ and Lemma \ref{BD_reg_D} in the case of a bounded domain provides a solution $\up\in\oPR\WSRper{2}{p}(\R; \LR{p}(\Omega)) \cap \oPR\WSRper{1}{p}\left(\R; \WSR{2}{p}(\Omega)\right)$ to
\begin{align}\label{Periodic}
\begin{pdeq}
\partial_t^2\up - \Delta\up - \lambda\partial_t\Delta\up &= \oPR f && \tin\R\times\Omega, \\
\up &= \oPR g && \ton\R\times\partial\Omega, \\
\up\left(t + \per, x\right) &= \up\left(t, x\right)
\end{pdeq}
\end{align}
that satisfies \eqref{est_2}. Setting $\uvel := \us + \up$, we thus obtain the desired solution to \eqref{damp_waveD}. Assume $v = v_s + v_p$ is another solution to \eqref{damp_waveD} with $v_s\in\WSRD{2}{q_1}\left(\Omega\right)$ and $v_p\in\oPR\WSRper{2}{q_2}(\R; \LR{q_2}(\Omega)) \cap \oPR\WSRper{1}{q_2}\left(\R; \WSR{2}{q_2}(\Omega)\right)$. Since $\up$ and $v_p$ both solve \eqref{Periodic}, the uniqueness statements of the lemmas mentioned above yield $\up = v_p$. Similarly, since both $\us$ and $v_s$ solve \eqref{Steady}, 
$\us - v_s$ is a polynomial of order 1 when $\Omega = \R^3$ or $\Omega = \R^3_+$, and $u_s - v_s = 0$ when $\Omega$ is a bounded domain.
\end{proof}

\begin{proof}[Proof of Theorem \ref{max_regN}]
From Lemma \ref{whole}, Lemma \ref{peri_Neumann} and Lemma \ref{BD_reg_N} we conclude in the same way as in the Dirichlet case the unique solvability of 
\begin{align}
\begin{pdeq}
\partial_t^2\up - \Delta\up - \lambda\partial_t\Delta\up &= \oPR f && \tin\R\times\Omega, \\
\pdn\up &= \oPR g && \ton\R\times\partial\Omega, \\
\up\left(t + \per, x\right) &= \up\left(t, x\right).
\end{pdeq}
\end{align}
If $\PR f$ and $\PR g$ satisfy the compatibility condition
\begin{align*}
\int_\Omega \PR f \, \dx + \int_{\partial\Omega} \PR g \, \dS = 0,
\end{align*}
\textit{i.e.}, if $f$ and $g$ satisfy the condition \eqref{CompCond}, standard theory on elliptic equations yields a solution $\us\in\WSRD{2}{p}\left(\Omega\right)$ to 
\begin{align*}
\begin{pdeq}
-\Delta\us &= \PR f && \tin\Omega, \\
\pdn\us &= \PR g && \ton\partial\Omega.
\end{pdeq}
\end{align*}
The uniqueness assertion of Theorem \ref{max_regN} follows as in the proof of Theorem \ref{max_regD}.
\end{proof}

\section{Nonlinear Problem -- The Kuznetsov Equation}

Existence of a solution to the nonlinear problems \eqref{KuznetsovD} and \eqref{KuznetsovN} shall now be established. We employ a fixed point argument based on the estimates established for the linearized systems \eqref{damp_waveD} and \eqref{damp_waveN} in the previous section.
\begin{thm}\label{max_regKD}
Assume that either $\Omega = \R^3$, $\Omega = \R^3_+$ or $\Omega\subset\R^3$ is a bounded domain with a $C^{1,1}$-smooth boundary. Let $p\in (\frac{5}{2}, 3)$. There is an $\varepsilon>0$ such that for all $f\in\LRper{p}(\R; \LR{p}\left(\Omega\right))$ and $g\in \TPDper(\R\times\partial\Omega)$ satisfying 
\begin{align}
\norm{f}_p + \norm{g}_{\TPD} \leq \varepsilon
\end{align}
there is a solution $\uvel$ to \eqref{KuznetsovD} with
\begin{align}
\uvel(t,x) = \us(x) + \up(t,x) \in \WSRD{2}{p}\left(\Omega\right)\oplus\oPR\WSRper{2}{p}(\R; \LR{p}(\Omega)) \cap \oPR\WSRper{1}{p}\left(\R; \WSR{2}{p}(\Omega)\right).
\end{align}
\end{thm}

\begin{thm}\label{max_regKN}
Let $\Omega$ and $p$ be as in Theorem \ref{max_regKD}. There is an $\varepsilon>0$ such that for all $f\in\LRper{p}(\R; \LR{p}\left(\Omega\right))$ and $g\in \TPNper(\R\times\partial\Omega)$ satisfying 
\begin{align}
\norm{f}_p + \norm{g}_{\TPN} \leq \varepsilon
\end{align}
and
\begin{align*}
\int_0^\per\int_\Omega f \,\dx\dt + \int_0^\per\int_{\partial\Omega} g \,\dS\dt = 0
\end{align*}
there is a solution $\uvel$ to \eqref{KuznetsovN} with
\begin{align}
\uvel(t,x) = \us(x) + \up(t,x) \in \WSRD{2}{p}\left(\Omega\right)\oplus\oPR\WSRper{2}{p}(\R; \LR{p}(\Omega)) \cap \oPR\WSRper{1}{p}\left(\R; \WSR{2}{p}(\Omega)\right).
\end{align}
\end{thm}

To prove Theorem \ref{max_regKD} and \ref{max_regKN}, we shall need estimates of the nonlinear terms in \eqref{KuznetsovD} and \eqref{KuznetsovN}. For this purpose we utilize the following embedding properties of time-periodic Sobolev spaces.
\begin{lem}\label{SobEmbeddingThm}
Let $\Omega$ be as in Theorem \ref{max_regKD} and $p\in(1,\infty)$.
Assume that $\alpha\in\bb{0,2}$ and $q_0,r_0\in[p,\infty]$ satisfy
\begin{align}\label{SobEmbeddingThm_Condp0}
\begin{pdeq}
&r_0\leq \frac{2p}{2-\alpha p} && \tif\ \alpha p<2,\\
&r_0<\infty && \tif\ \alpha p =2,\\
&r_0\leq\infty && \tif\ \alpha p >2,
\end{pdeq}
\qquad
\begin{pdeq}
&q_0\leq \frac{np}{n-(2-\alpha) p} && \tif\  \np{2-\alpha}p<{n},\\
&q_0<\infty && \tif\ \np{2-\alpha}p={n},\\
&q_0\leq\infty && \tif\ \np{2-\alpha}p>{n},
\end{pdeq}
\end{align}
and that $\beta\in\bb{0,1}$ and $q_1,r_1\in[p,\infty]$ satisfy
\begin{align}\label{SobEmbeddingThm_Condp1}
\begin{pdeq}
&r_1\leq \frac{2p}{2-\beta p} && \tif\ \beta p<2,\\
&r_1<\infty && \tif\ \beta p =2,\\
&r_1\leq\infty && \tif\ \beta p >2,
\end{pdeq}
\qquad
\begin{pdeq}
&q_1\leq \frac{np}{n-(1-\beta) p} && \tif\  \np{1-\beta}p<{n},\\
&q_1<\infty && \tif\ \np{1-\beta}p={n},\\
&q_1\leq\infty && \tif\ \np{1-\beta}p>{n}.
\end{pdeq}
\end{align}
Then for all $\uvel\in\WSRper{1,2}{p}(\R\times\Omega) := \WSRper{1}{p}(\R; \LR{p}(\Omega)) \cap \LRper{p}\left(\R; \WSR{2}{p}(\Omega)\right)$:
\begin{align}\label{SobEmbeddingThm_Est}
\norm{\uvel}_{\LRper{r_0}\np{\R;\LR{q_0}(\Omega)}}
+\norm{\grad\uvel}_{\LRper{r_1}\np{\R;\LR{q_1}(\Omega)}}  \leq
\Cc{C}\norm{\uvel}_{1,2,p},
\end{align}
with $\Cclast{C}=\Cclast{C}(\per,n,\Omega,r_0,q_0,r_1,q_1)$.
\end{lem}
\begin{proof}
See \cite[Theorem 4.1]{GaldiKyed}.
\end{proof}
Furthermore, we make use of the following lemma.

\begin{lem}\label{nonlin}
Let $\Omega$ and $p$ be as in Theorem \ref{max_regKD}. Then 
\begin{align*}
\norm{\partial_t v\partial_t^2\uvel}_p + \norm{\grad v\cdot\partial_t\grad\uvel}_p \leq c\norm{v}_{\SolS{p}}\norm{\uvel}_{\SolS{p}}
\end{align*}
holds for any $\uvel, v\in\SolS{p}\left(\torus\times\Omega\right)$. 
\end{lem}
\begin{proof}
Clearly, $\partial_t v\in\WSR{1}{p}\left(\torus; \LR{p}\left(\Omega\right)\right)\cap\LR{p}\left(\torus; \WSR{2}{p}\left(\Omega\right)\right)$ for any $v\in\SolS{p}\left(\torus\times\Omega\right)$. Using \eqref{SobEmbeddingThm_Condp0} with $\alpha = \frac{4}{5}$, we deduce for $p\in (\frac{5}{2}, \infty)$ 
\begin{align*}
\norm{\partial_t v}_{\infty} \leq c\norm{v}_{\SolS{p}}.
\end{align*}
It thus follows from H\"older's inequality for $p\in(\frac{5}{2}, \infty)$ that
\begin{align}\label{nonlin_1}
\norm{\partial_t v\partial_t^2\uvel}_p \leq \norm{\partial_t v}_{\infty}\norm{\partial_t^2\uvel}_{p} \leq c\norm{\partial_t v}_{1,2,p}\norm{u}_{\SolS{p}} \leq c\norm{v}_{\SolS{p}}\norm{\uvel}_{\SolS{p}}.
\end{align}
A further application of H\"older's inequality yields
\begin{align*}
\norm{\grad v\cdot\partial_t\grad\uvel}_p \leq \norm{\grad v}_{\LR{\infty}\left(\torus; \LR{3}\left(\Omega\right)\right)} \norm{\partial_t\grad\uvel}_{\LR{p}(\torus; \LR{\frac{3p}{3-p}}\left(\Omega\right))}.
\end{align*}
From Lemma \ref{SobEmbeddingThm} with $\beta=0$ we obtain for $p\in(1, 3)$ 
\begin{align*}
\norm{\partial_t\grad\uvel}_{\LR{p}(\torus; \LR{\frac{3p}{3-p}}\left(\Omega\right))}\leq c\norm{\partial_t\uvel}_{1,2,p} \leq c\norm{\uvel}_{\SolS{p}}.
\end{align*}
Choosing $\beta = \frac{4}{5}$, Lemma \ref{SobEmbeddingThm} yields for all $p\in (\frac{5}{2}, \infty)$
\begin{align*}
\norm{\grad v}_{\LR{\infty}\left(\torus; \LR{3}\left(\Omega\right)\right)} \leq c\norm{v}_{\SolS{p}}.
\end{align*}
Hence we obtain for $p\in (\frac{5}{2}, 3)$
\begin{align}\label{nonlin_2}
\norm{\grad v\cdot\partial_t\grad\uvel}_p \leq c\norm{v}_{\SolS{p}}\norm{\uvel}_{\SolS{p}}.
\end{align}
The lemma now follows from \eqref{nonlin_1} and \eqref{nonlin_2}.
\end{proof}

\begin{proof}[Proof of Theorem \ref{max_regKD}]
We shall establish existence of a solution $u$ to \eqref{KuznetsovD} of the form $u = \us + \up$, where $\us\in\WSRD{2}{p}\left(\Omega\right)$ is a solution to the steady-state problem 
\begin{align}\label{steady_waveD}
\begin{pdeq}
-\Delta\us &= \PR f && \tin\Omega, \\
\us &= \PR g && \ton\partial\Omega
\end{pdeq}
\end{align}
and $\up\in\oPR\WSRper{2}{p}(\R; \LR{p}(\Omega)) \cap \oPR\WSRper{1}{p}\left(\R; \WSR{2}{p}(\Omega)\right)$ a solution to the purely periodic problem
\begin{align}\label{periodic_waveD}
\begin{pdeq}
\partial_{t}^2\up-\Delta\up - \lambda\partial_t\Delta\up - \partial_t\bp{\gamma\np{\partial_t\up}^2+\snorm{\grad\up}^2} \\
- 2\grad\us\cdot\grad\partial_t\up &= \oPR f && \tin\R\times\Omega, \\
\up &= \oPR g && \ton\R\times\partial\Omega.
\end{pdeq}
\end{align}
Standard theory for elliptic problems yields for every $\PR f\in \LR{p}\left(\Omega\right)$ a solution $\us\in\WSRD{2}{p}\left(\Omega\right)$ to \eqref{steady_waveD} with
\begin{align}\label{steadyD}
\norm{\grad\us}_{\frac{3p}{3-p}} \leq \norm{\grad^2\us}_p \leq c\left(\norm{\PR f}_p + \norm{\PR g}_{\TPD}\right) \qquad \forall p\in (1, 3).
\end{align}
The solution to \eqref{periodic_waveD} shall be obtained as a fixed point of the mapping
\begin{align*}
&\mathcal{N}\colon\oPR\WSRper{2}{p}(\R; \LR{p}(\Omega)) \cap \oPR\WSRper{1}{p}\left(\R; \WSR{2}{p}(\Omega)\right) \\
&\qquad\qquad\qquad\qquad\qquad\to \oPR\WSRper{2}{p}(\R; \LR{p}(\Omega)) \cap \oPR\WSRper{1}{p}\left(\R; \WSR{2}{p}(\Omega)\right)\\
&\mathcal{N}\left(\up\right) := \operatorname{A}^{-1}\left(\partial_t\left(\gamma\left(\partial_t\up\right)^2 + \left|\grad\up\right|^2\right) + 2\grad\us\cdot\grad\partial_t\up + \oPR f, \oPR g\right)
\end{align*}
with $\operatorname{A}$ as in Corollary \ref{homeo}.
We shall verify that $\mathcal{N}$ is a contracting self-mapping on a ball of sufficiently small radius. For this purpose, let $\rho > 0$ and consider some $\up\in \oPR\WSRper{2}{p}(\R; \LR{p}(\Omega)) \cap \oPR\WSRper{1}{p}\left(\R; \WSR{2}{p}(\Omega)\right)\cap B_\rho$. 
Since $\operatorname{A}$ is a homeomorphism, we obtain
\begin{align*}
\norm{\mathcal{N}\left(\up\right)}_{\SolS{p}} &c\,\leq c\,\norm{\operatorname{A}^{-1}} \ \Big(\norm{\partial_t\up\partial_t^2\up}_p + \norm{\grad\up\cdot\partial_t\grad\up}_p + \norm{\grad\us\cdot\grad\partial_t\up}_p \\ &\qquad+ \norm{\oPR f}_p + \norm{\oPR g}_{\TPD}\Big).
\end{align*}
Utilizing Lemma \ref{nonlin}, we find that
\begin{align*}
\norm{\partial_t\up\partial_t^2\up}_p + \norm{\grad\up\cdot\partial_t\grad\up}_p \leq c\,\norm{\up}_{\SolS{p}}^2. 
\end{align*}
Employing \eqref{steadyD} and Lemma \ref{SobEmbeddingThm} with $\beta=0$, we also obtain
\begin{align*}
\norm{\grad\us\cdot\grad\partial_t\up}_p 
\leq \norm{\grad\us}_{\LR{\frac{3p}{3-p}}(\Omega)}\norm{\grad\partial_t\up}_{\LRper{p}\np{\R;\LR{3}(\Omega)}}
\leq c\,\norm{\grad^2\us}_{p}\norm{\up}_{\SolS{p}}.
\end{align*}
Consequently, 
\begin{align*}
\norm{\mathcal{N}\left(\up\right)}_{\SolS{p}}
& \leq c\left(\rho^2 + \varepsilon\rho + \varepsilon\right).
\end{align*}
Choosing $\varepsilon = \rho^2$ and $\rho$ sufficiently small, we have $c\left(\rho^2 + \varepsilon\rho + \varepsilon\right) \leq \rho$, \textit{i.e.}, $\mathcal{N}$ is a self-mapping on $B_\rho$. Moreover
\begin{align*}
\norm{\mathcal{N}&\left(\up\right) - \mathcal{N}\left(v_p\right)}_{\SolS{p}} \leq c\,\norm{\operatorname{A}^{-1}} \ \Big(\norm{\partial_t\up\partial_t^2\up - \partial_t v_p\partial_t^2 v_p}_p \\
&\quad + \norm{\grad\up\cdot\partial_t\grad\up - \grad v_p\cdot\partial_t\grad v_p}_p + \norm{\grad\us\cdot\partial_t\grad\up - \grad\us\cdot\partial_t\grad v_p}_p\Big) \\
&\leq c\Big(\norm{\partial_t\up\partial_t^2\left(\up - v_p\right)}_p + \norm{\partial_t^2 v_p\partial_t\left(\up - v_p\right)}_p + \norm{\grad\up\cdot\partial_t\grad\left(\up - v_p\right)}_p \\ 
&\quad + \norm{\partial_t\grad v_p\cdot\grad\left(\up - v_p\right)}_p + \norm{\grad\us\cdot\partial_t\grad\left(\up - v_p\right)}_p\Big) \\
&\leq c\left(4\rho\norm{\up - v_p}_{\SolS{p}} + \varepsilon\norm{\up - v_p}_{\SolS{p}}\right) = c(4\rho + \rho^2) \norm{\up - v_p}_{\SolS{p}}.
\end{align*}
Therefore, if $\rho$ is sufficiently small $\mathcal{N}$ becomes a contracting self-mapping. By the contraction mapping principle, existence of a fixed point for $\mathcal{N}$ follows. This concludes the proof.
\end{proof}

\begin{proof}[Proof of Theorem \ref{max_regKN}]
Similar to the proof of Theorem \ref{max_regKD}.
\end{proof}

\bibliographystyle{abbrv}

\end{document}